\RequirePackage{fix-cm}
\documentclass[smallextended]{svjour3}       
\smartqed  
\usepackage{graphicx}
\usepackage{amsfonts}
\usepackage{mathrsfs}
\usepackage{amsmath}
\usepackage [latin1]{inputenc}
\begin{document}

\title{ON  CONTROLLED INVARIANCE OF REGULAR DISTRIBUTIONS
}


\author{Qianqian Xia          
}


\institute{Qianqian Xia \at
            Nanjing University of Information Science \& Technology, Nanjing, 210044, China\\
              \email{inertialtec@sina.com}
}

\date{Received: date / Accepted: date}

\maketitle

\begin{abstract}
This paper considers the problem of controlled invariance of involutive regular distribution,  both for smooth and real analytic cases. After a review of some existing work, a precise formulation of the problem of local and  global controlled invariance of involutive regular distributions for both affine control systems and affine distributions is introduced. A complete characterization for local controlled invariance of involutive regular distributions for affine control systems is presented. A geometric interpretation for this characterization is provided. A result on local controlled invariance for real analytic affine distribution is given.  Then we investigate conditions that allow passages from local controlled invariance to global controlled invariance, for both smooth and real analytic affine distributions. We clarify existing results in the literature. Finally, for manifolds with a symmetry Lie group action, the problem of global controlled invariance is considered.
\keywords{ Controlled invariance\and Affine control system\and Generalized distribution\and Local\and  Global \and Real analyticity}
\end{abstract}

\section{Introduction}
\label{intro}
Let $r \in \{\infty, \omega\}$. Throughout the paper, all manifolds are assumed to be paracompact Hausdorff $C^r$-manifolds if not specified.  Let $L$ be a $C^r$-distribution on a $C^r$-manifold $M$, i.e. a subbundle of the tangent bundle $TM$, where the dimension of the the subspace $L_p$ is a constant function of $p \in M$.

In control theory one encounters generalized distributions \cite{sussmann,Stefan,Liu,Bullo}, where the subspace $L_p$ can have different dimensions at different points. We call a distribution of a constant rank a regular distribution.  A generalized distribution $L$ is of class $C^r$, if for each point $p \in M$, there exist local $C^r$-sections of $L$ whose values at $p$ span  $L_p$. We refer to Section \ref{sec:2} for a detailed discussion on generalized distributions.

Consider the affine control system $\Sigma$
\begin{equation}\label{eq:1.1}
\dot x=f(x)+\sum_{i=1}^{m}g_i(x)u_i
\end{equation}
on $M$, where $f, g_1, \cdots, g_m$ are $C^r$-vector fields on $M$.

Then associated to $\Sigma$ there is a  $C^r$-generalized distribution $G$ of $M$,  defined by
\begin{equation}\label{eq:1.4}
 G_p=\text{span}_{\mathbb{R}}\{g_1(p), \cdots, g_m(p)\},
\end{equation}
for $p \in M$, where $g_1, \cdots, g_m$ are said to be the $C^r$-generators for $G$.

The following definition is taken from \cite{1}.
\begin{definition}\label{def:1.1}
Let $U$ be an open subset of $M$. A pair of $C^r$-functions $(\alpha, \beta)$, where $\alpha: U \rightarrow \mathbb{R}^{m}$ and $\beta:  U\rightarrow GL(m, \mathbb{R})$, are called feedback functions on $U$. When $U=M$ we refer to feedback functions as global feedback. If
$\alpha=(\alpha_1 \cdots \alpha_m)^T$ and $\beta=(\beta_{ij})_{m\times m}$ are feedback functions on $U$, then the new system obtained by feedback,
\begin{equation}\label{eq:1.2}
\dot x=(f+\sum_{i=1}^{m}g_i \alpha_i)+\sum_{i=1}^{m}(\sum_{j=1}^m\beta_{ij}g_j)v_i=\hat f(x)+\sum_{i=1}^{m}\hat{g}_i(x)v_i,
\end{equation}
is said to be pure feedback equivalent to $\Sigma$.
\end{definition}
Let $D$ be an involutive regular  $C^r$-distribution of $M$.   In the previous work there is the following definition of local controlled invariance \cite{1,Hirschorn,Krener,Nijmeijer,Ikggm,Byrnes,Schaft,schaft,Isidory,Cheng}.
\begin{definition}\label{def:1.2}
An involutive regular $C^r$-distribution $D$ of $M$ is said to be locally controlled invariant for $\Sigma$ at $p \in M$, if there exist a sufficiently small neighborhood $U$ of $p$, and $C^r$-feedback functions $(\alpha, \beta)$ on $U$ such that
\begin{subequations}\label{eq:1.3}
\begin{align}
&[\hat f, \Gamma^r(U;D)] \subseteq \Gamma^r(U;D), \label{eq:1.3a} \\
&[\hat g_i, \Gamma^r(U;D)]\subseteq \Gamma^r(U;D), i=1, \cdots, m, \label{eq:1.3b}
\end{align}
\end{subequations}
where $\hat f, \hat g_i,i=1, \dots, m,$ are given by $(\ref{eq:1.2})$,  and $\Gamma^r(U;D)$ denotes the space of $C^r$-sections of $D$ over $U$.
\end{definition}
Let $\text{span}_{C^r(U)}\{g_1, \cdots, g_m\}$ be the module over $C^r(U)$ spanned by $g_1, \cdots, g_m$, where $C^r(U)$ denotes the ring of $C^r$-functions on the open subset $U$. From \eqref{eq:1.3a} and \eqref{eq:1.3b} it is obvious that a necessary condition for local controlled invariance  is that $D$ satisfies
\begin{subequations}
\begin{align}
&[f, \Gamma^r(U;D)] \subseteq \Gamma^r(U;D)+\text{span}_{C^{r}(U)}\{g_1, \cdots, g_m\}, \label{eq:1.8a} \\
&[g_j, \Gamma^r(U;D)]\subseteq \Gamma^r(U;D)+\text{span}_{C^{r}(U)}\{g_1, \cdots, g_m\}, j=1, \cdots, m. \label{eq:1.8b}
\end{align}
\end{subequations}

The notion of controlled invariant subspaces in the linear systems theory dates back to the end of sixties \cite{Basile,Wonham}. A modern account of the use of controlled and conditioned invariant subspaces is given in \cite{wonham}. The nonlinear generalization of the notion of controlled
invariance, which plays a crucial role in disturbance decoupling problems,  has been initiated in \cite{Hirschorn,Krener}. See also \cite{Isidory,MW,isidory,Moog,Krener2,Respondek}. There has been a lot of work on conditions that ensure local controlled invariance  for the $C^{\infty}$-affine control system $\Sigma$, mainly focusing on the following conditions:
\begin{subequations}
\begin{align}
&[f, D](q)\subseteq (D+G)(q), \label{eq:1.4a} \\
&[g_j, D](q)\subseteq (D+G)(q), j=1, \cdots, m.\label{eq:1.4b}
\end{align}
\end{subequations}
for each $q$ in a sufficiently small neighborhood $U$ of $p \in M$. In \cite{Schaft},  Theorem 7.5 states that if $G$ and $D\cap G$ have constant rank, then \eqref{eq:1.4a} and \eqref{eq:1.4b} are sufficient to ensure that  $D$ is locally controlled invariant for $\Sigma$. In \cite{1,Hirschorn,Krener,Nijmeijer,Ikggm,Byrnes,schaft,Isidory}, the above conditions are replaced by only requiring that $G+D$ has constant rank, or equivalently $(G+D)/D$ has constant rank, then the sufficiency of  \eqref{eq:1.4a} and \eqref{eq:1.4b} is ensured. This is known as Quaker Lemma. In \cite{Cheng}, a necessary and sufficient condition concerning the $C^{\infty}(U)$-module spanned by $g_1/D, \cdots, g_m/D$ is presented for the local controlled  invariance of the involutive regular distribution $D$, where no additional condition of constant rank is required. We refer to Remark \ref{remark 1} in Section \ref{sec:3} for a precise statement of this result and its equivalence with our result. However, the proof provided in \cite{Cheng}
is not a geometric one.

In this paper, by exploring the underlying geometric structure, we will show that instead of \eqref{eq:1.4a} and \eqref{eq:1.4b}, \eqref{eq:1.8a} and \eqref{eq:1.8b} are necessary and sufficient for local controlled invariance for $C^r$-affine control system $\Sigma$. We will provide an intrinsic and complete differential geometric proof for our result. Differential geometric interpretations for the feedback functions $(\alpha, \beta)$ required in Definition \ref{def:1.2} will be given. A previous work using geometric methods has been done in \cite{1} on the proof of Quaker Lemma, where connection and parallel translation are introduced. Although  the result there is not complete (we will give a detailed explanation for this in Section \ref{sec:3}), a geometric intuition is  provided which inspires our consideration here.

One thing that should be pointed out is that it can not  be assumed that any  $C^r$-local generators for a $C^r$-generalized distribution can span  the module of  $C^r$-local sections of the  $C^r$-generalized distribution \cite{Lewis,Drager}. Hence in general the conditions of \eqref{eq:1.8a} and \eqref{eq:1.8b} is stronger than  \eqref{eq:1.4a} and \eqref{eq:1.4b}. This  observation  leads  to the following definition of local controlled invariance for generalized affine distribution, which  is  equivalent with Definition \ref{def:1.2} if $p$ is a regular point for $G$ and is a new one if  $p$ is a singular point for $G$.
\begin{definition}\label{def:1.3}
An involutive regular $C^r$-distribution $D$ of $M$ is said to be locally  controlled invariant for a $C^r$-generalized affine distribution $\Delta$ at $p$, if there exist a sufficiently small local neighborhood $U$ of $p$ and $C^r$-vector fields $f,  g_i, i \in I$ on $U$, such that  $[f, \Gamma^r(U;D)]\subseteq \Gamma^r(U;D)$ and $[g_i, \Gamma^r(U;D)] \subseteq \Gamma^r(U;D), i \in I$, where $\Delta_q=f(q)+\text{span}_{\mathbb{R}}\{g_i(q)|i \in I\}$ for each $q\in U$, and $I$ is an index set.
\end{definition}

We will provide a sufficient condition for local controlled invariance for $C^{\omega}$-generalized affine distribution in Section \ref{sec:3}.

Another main problem that we will investigate in this paper is the problem of global controlled invariance. We have the following definitions for affine distributions and affine control systems respectively.
\begin{definition}\label{def:1.4}
An  involutive regular $C^r$-distribution $D$ of $M$ is said to be  globally  controlled invariant for a  $C^r$-generalized affine distribution $\Delta$, if there exist $C^r$-vector fields $ f, g_i, i\in I$ on $M$,  such that
\begin{subequations}
\begin{align}
&[f, \Gamma^r(D)]\subseteq \Gamma^r(D),\label{eq:1.7a}\\
&[g_i, \Gamma^r(D)] \subseteq \Gamma^r(D), i \in I.\label{eq:1.7b}
\end{align}
\end{subequations}
where $\Delta_p= f(p)+\text{span}_{\mathbb{R}}\{ g_i(p)| i \in I\}$ for all $p \in M$, and $I$ is an index set.
\end{definition}
\begin{definition}
 An  involutive regular  $C^r$-distribution $D$ of $M$ is said to be globally controlled invariant for $\Sigma$, if there exist global feedback  $(\alpha, \beta)$ such that
\begin{subequations}
\begin{align}
&[\hat f, \Gamma^r(D)] \subseteq \Gamma^r(D),\label{eq:1.6a}\\
&[\hat g_i, \Gamma^r(D)]\subseteq \Gamma^r(D), i=1, \cdots, m,\label{eq:1.6b}
\end{align}
where $\hat f, \hat g_i, i=1, \cdots, m$, are given by $(\ref{eq:1.2})$ with $U=M$.
\end{subequations}
\end{definition}

Since $D$ has constant rank,  there exist  generators $d_1, \cdots, d_k \in \Gamma^r(U;D)$ for the $C^r(U)$-module of $\Gamma^r(U;D)$ of $D$. Besides, $d_1|V, \cdots, d_k|V \in \Gamma^r(V;D)$ are generators for the $C^r(V)$-module of $\Gamma^r(V;D)$ of $D$, where $V$ is any open subset of $U$. Since  $M$ is paracompact and Hausdorff, then it follows from partition of unity  \cite{Abraham} in the smooth case and Cartan's Theorem A \cite{Cartan} in the real analytic case that, for each $p \in M$,  there exist generators for the module of local sections of $D$ on a neighborhood of $p$ that are the restrictions of global sections of $D$ to this neighborhood. Then by restrictions of  globally defined vector fields in \eqref{eq:1.7a}, \eqref{eq:1.7b}, \eqref{eq:1.6a}, and \eqref{eq:1.6b} to the above neighborhood of $p$, it follows immediately that global controlled invariance implies the corresponding local controlled invariance.  However, if $D$ does not have constant rank, complications may arise. A previous work on controlled invariance of singular distribution can be found in \cite{Rama}.

It is obvious that in general  we can not infer from local controlled invariance at every point in $M$ that global controlled invariance holds. In this paper, we will investigate cases and conditions that admit global controlled invariance.  Note that \cite{1} studies a special case of global controlled invariance for affine control systems, where $g_1, \cdots, g_m$ are assumed to be globally defined linearly independent vector fields. In  Section \ref{sec:4} we will investigate the problem of global controlled invariance for generalized affine distributions from the following two aspects: (1) special geometric structure of the state space $M$ that admits global controlled invariance; (2) conditions that allow passages from local controlled invariance to global controlled invariance. Both require techniques  different from \cite{1}.  We will also review  result in \cite{1} and do some possible generalization.  We refer to \cite{1} for detailed information about existing work on global controlled invariance.

In summary, the main contribution of the paper is two fold: (1) We provide a differential geometric interpretation for the problem of local controlled invariance of involutive regular distribution for $C^r$-affine control system. (2) We provide conditions that allow  global controlled invariance for (i) $C^r$-generalized affine distribution of general $C^r$-manifold; (ii) smooth generalized affine distribution of smooth manifold with a symmetry Lie group action.

The paper is organized as follows. Section \ref{sec:2} provides some basic definitions, results and facts on generalized distributions.
In Section \ref{sec:3} we study local controlled invariance for both affine control system and affine distribution.
The corresponding problems of global controlled invariance are investigated  in  Section \ref{sec:4},  and the conclusions follow in Section \ref{sec:5}.
\section{Generalized distributions}
\label{sec:2}
In this section we introduce some basic notions and results on generalized distributions and generators used in the paper.  We refer to \cite{Lewis} for a comprehensive and elegant treatment on this topic.
\begin{definition}
Let $M$ be a $C^r$-manifold. A generalized distribution of $M$ is  a subset $G \subseteq TM$ such that for each
$p\in M$, the subset $G_p\subseteq T_pM$ is a subspace of $T_pM$.
A generalized affine distribution $\Delta$ of $M$ is an assignment of an affine subspace $\Delta_p\subseteq T_pM$ for each $p \in M$.
Associated with $\Delta_p$ is a subspace of $T_pM$. The distribution
which assigns to $p$ this subspace is the linear part of $\Delta$ and is denoted by $L(\Delta)$. We do not assume the subspaces $G_p$ or $L(\Delta_p)$ vary continuously with $p$ or have constant dimension.

A generalized distribution $G$ is of class $C^{r}$,  if, for each $p \in M$, there exist a neighbourhood $U$ of $p$ and a family $(g_j)_{j \in J}$ of $C^r$-sections, called local
generators, of $TM|U$ such that
\begin{equation}\label{eq:2.3}
G_q= \text{span}_{\mathbb{R}}\{g_j(q)|j \in J\},
\end{equation}
for each $q \in U$. A generalized affine distribution $\Delta$ is of class $C^r$, if, for each $p \in M$, there exist a neighbourhood $U$ of $p$ and a family $(h_0, (h_i)_{i \in I})$ of $C^r$-sections, called local
generators, of $TM|U$ such that
\begin{equation}\label{eq:2.3}
\Delta_q= h_0(q)+\text{span}_{\mathbb{R}}\{h_i(q)|i \in I\},
\end{equation}
for each $q \in U$.

A generalized distribution $G$  of class $C^{r}$ is locally finitely generated if, for each $p \in M$, there exist a neighbourhood $U$ of $p$ and a family $(g_1, \cdots, g_k)$ of
$C^r$-sections, called local generators, of $TM|U$ such that
\begin{equation}
G_q = \text{span}_{\mathbb{R}}\{g_1(q), \cdots, g_k(q)\},
\end{equation}
for each $q \in U$. A generalized affine distribution $\Delta$  of class $C^{r}$ is locally finitely generated if the associated generalized distribution $L(\Delta)$ of class $C^{r}$  is locally finitely generated.

A generalized distribution $G$  of class $C^{r}$,  is globally finitely generated if,  there exists  a family $(g_1, \cdots, g_k)$ of
$C^r$-sections, called global generators, of $TM$ such that for each $p\in M$,
\begin{equation}
G_p = \text{span}_{\mathbb{R}}\{g_1(p), \cdots, g_k(p)\}.
\end{equation}
The nonnegative integer $\dim(G_p)$ at $p$ is called the rank of $G$ at $p$ and is sometimes denoted by $\text{rank}(G_p)$.
A generalized affine distribution $\Delta$  of class $C^{r}$,  is globally finitely generated if,  there exists  a family $(h_0, h_1, \cdots, h_l)$ of
$C^r$-sections, called global generators, of $TM$ such that for each $p\in M$,
\begin{equation}
\Delta_p = h_0(p)+\text{span}_{\mathbb{R}}\{h_1(p), \cdots, h_l(p)\}.
\end{equation}
\end{definition}

\begin{definition}
Let $M$ be a $C^r$-manifold. Let $G$ be a $C^r$-generalized
distribution of $M$ and $\Delta$ be a $C^r$-generalized affine distribution of $M$.
If $U\subseteq M$ is open, a local section of $G$ or $\Delta$ over $U$ is a section $g: U \rightarrow TM$ such that $g(p) \in G_p$ or $g(p) \in \Delta_p$ for every $p \in U$. A local section $g$ of $G$ or $\Delta$ is of class $C^{r}$, if it is of class $C^{r}$ as a local section of $TM$. The set of local sections of $G$ or $\Delta$ over $U$ of class $C^{r}$ is denoted by $\Gamma^r(U;G)$ or $\Gamma^r(U;\Delta)$,
or simply by $\Gamma^r(G)$ or $\Gamma^r(\Delta)$ when $U = M$.
\end{definition}

It follows immediately that $\Gamma^r(U;G)$ is a module over $C^r(U)$.
An equivalent characterization of $C^{r}$-generalized distribution or $C^r$-generalized affine distribution is given as follows:
a generalized distribution $G$ or a generalized affine distribution $\Delta$ of $M$ is of class $C^{r}$,  if and only if for every point $p$ and every $v\in G_{p}$ or $\Delta_p$, there exist a neighbourhood $U$ of $p$ and a $C^{r}$-section $s \in \Gamma(U; G)$  or $\Gamma(U; \Delta)$ such that $s_{p}=v$.

If $G$ is a  smooth generalized distribution of a connected smooth manifold $M$, the following result ensures that $G$ is globally finitely generated.
\begin{theorem}\cite{Sussmann,Drager}\label{tm:2.4}
Let $M$ be a connected smooth  manifold. Let $G$ be a smooth generalized distribution of $M$. Then $G$ is globally finitely generated.
\end{theorem}

If $G$ is a $C^{\omega}$-generalized distribution of a $C^{\omega}$-manifold $M$, the following results ensure that $G$ is locally  finitely generated. Besides, there exist sufficiently small neighborhood $U$ of $p \in M$, such that the module $\Gamma^{\omega}(U; G)$ over $C^{\omega}(U)$
is finitely generated.

\begin{theorem}\cite{Lewis}\label{tm:2.3}
Let $M$ be a real analytic manifold and $G$ be a generalized distribution of $M$, then the following statements are equivalent:\\
1. $G$ is real analytic.\\
2. For each $p \in M$ and each $v_p \in G_p$, there exists a real analytic section $g\in \Gamma^{\omega}(G)$
over $M$ such that $g_p=v_p$.\\
3. For each $p\in M$ there exist a neighbourhood $\mathcal{N}$ of $p$ and real analytic sections
$g_1, \cdots, g_k \in \Gamma^{\omega}(G)$ over $M$ such that
\[G_q=\text{span}_{\mathbb{R}}\{g_1(q), \cdots, g_k(q)\},\]
for each $q \in \mathcal{N}$.
\end{theorem}
\begin{proposition}\cite{Lewis}\label{pro:2.13}
Let $M$ be a real analytic manifold and $G$ be a real analytic generalized distribution of $M$.  Then for each $p \in M$, there exist a neighbourhood
$ U$ of $p$ and $s_1, \cdots, s_l \in \Gamma^{\omega}(U; G)$ such that $s_1|V, \cdots, s_l|V$ are generators for the module $\Gamma^{\omega}(V; G)$ over $C^{\omega}(V)$, for every open set $V \subseteq U$.
\end{proposition}

One of the complications ensuing from the notion of a generalized distribution arises if the dimensions of the subspaces
$G_p, p \in M$,  are not locally constant. The following definition associates some language with
this.
\begin{definition}
Let $M$ be a $C^r$-manifold and $p$ be point in $M$. Let $G$ be a $C^r$-generalized
distribution of $M$.  $p$ is a regular point for $G$ if there exists a neighbourhood $U$  of $p$ such that the dimension of the linear subspace $G_q$ is equal to that of $G_{p}$,  for every $q\in U$.  $p$ is a singular point for $G$ if it is not a regular point for $G$.

A generalized distribution $G$ is regular if every point in $M$ is a regular point for $G$, and is singular otherwise.
\end{definition}

Let $p$ be a regular point for the $C^r$-generalized distribution $G$ with $\text{rank}(G_p)=k$. From (\ref{eq:2.3}) we know that there exist local $C^r$-sections $g_1, \cdots, g_k$ of $G$ such that $G_p=\text{span}_{\mathbb{R}}\{g_1(p), \cdots, g_k(p)\}$. Since $p$ is a regular point, it follows that $G_q=\text{span}_{\mathbb{R}}\{g_1(q), \cdots, g_k(q)\}$ for $q \in U$, where $U$ is any sufficiently small  neighborhood of $p$. That is, $G$ is finitely generated over $U$. Besides, let $h_1, \cdots, h_l$ be any local generators of $G$ over $U$.  $\Gamma^r(U;G)$ is a module over $C^r(U)$ generated by $h_1, \cdots, h_l$. That is, we can identify a $C^r$-distribution, the module spanned by
any $C^r$-local generators, and the space of its local $C^r$-sections around a regular point.

If $p$ is a singular point for a $C^r$-generalized distribution $G$, the following two examples show that: (1) If $r=\infty$, there exists
arbitrarily small neighborhood $U$ of $p$ such that $\Gamma^{\infty}(U; G)$ is not finitely generated.
(2) If $r=\omega$, there exist local generators of $G$ which can not span the module $\Gamma^{\omega}(U;G)$, where $U$ is any  local neighborhood of $p$.
\begin{example}\cite{Drager}
Let $M=\mathbb{R}$ and $TM=\mathbb{R}\times \mathbb{R}$. Consider the generalized distribution
$G$ defined by
\begin{equation}
G_p=\begin{cases}
\{p\}\times \mathbb{R}, &p \textgreater 0, \\
\{p\}\times \{0\}, &p\leq 0. \nonumber
\end{cases}
\end{equation}
Consider the smooth section $g $ of $G$ defined by
\begin{equation}
g_p=\begin{cases}
(p,e^{-1/p}), &p \textgreater 0, \\
(p,0), &p\leq 0. \nonumber
\end{cases}
\end{equation}
It follows that $G$ is a smooth distribution of $M$ since $g$ is a global generator for $G$. However, \cite{Drager}
shows that $\Gamma^{\infty}(U; G)$ is not finitely generated, for any neighborhood $U=(-a, a), a \textgreater 0$ of the singular point $0$ for $G$.
\end{example}
\begin{example}\cite{Lewis}
Let $M=\mathbb{R}$ and $TM=\mathbb{R}\times \mathbb{R}$. Consider the $C^{\omega}$-generalized distribution
$G$ defined by
\begin{equation}
G_p=\begin{cases}
\{p\}\times \mathbb{R}, &p \neq 0, \\
\{p\}\times \{0\}, &p=0. \nonumber
\end{cases}
\end{equation}
It has a global generator  $g(p)=(p, p^2)$. However, consider the $C^{\omega}$-section
$h(p)=(p, p)$ of $G$. It is obvious that $h|U$ is not  a $C^{\omega}(U)$-multiple of $g|U$, where $U\subseteq M$ is any open subset containing the singular point $0$ of $G$.
\end{example}

\section{Local controlled invariance}
\label{sec:3}
We have the following result for local controlled invariance.
\begin{theorem}\label{tm:3.5}
Consider the $C^r$-affine control system \eqref{eq:1.1} on $M$.  Let $D$ be an involutive regular  $C^r$-distribution of $M$. $D$ is locally controlled invariant for $\Sigma$ at $p \in M$, if and only if there exists a sufficiently small neighborhood $U$ of $p$, such that
\begin{subequations}
\begin{align}
&[g_i,\Gamma^r(U; D)]\subseteq \Gamma^r(U;D)+\text{span}_{C^{r}(U)}\{g_1, \cdots, g_m\}, i=1, \cdots, m. \label{eq:3.9a}\\
&[f, \Gamma^r(U;D)] \subseteq \Gamma^r(U;D)+\text{span}_{C^{r}(U)}\{g_1, \cdots, g_m\}, \label{eq:3.9b}
\end{align}
\end{subequations}
where $\text{span}_{C^r(U)}\{g_1, \cdots, g_m\}$ denotes the $C^r(U)$-module spanned by $g_1, \cdots, g_m$.
\end{theorem}

We will propose a new proof based on the notion
of connection. Actually, an involutive regular distribution $D$ of $M$ defines a connection on $TM/D$ as claimed
in \cite{1}. Consider the  vector bundle $\rho: TM/D\rightarrow M$. Let $\pi: TM\rightarrow TM/D$ be the projection map. Define a map $\nabla$ as follows:
\begin{eqnarray}\label{eq:2.1}
\nabla:
&\Gamma^r(D)\times\Gamma^r(TM/D)\rightarrow \Gamma^r(TM/D)\nonumber\\
&\nabla_{X}\bar Y(p)=\pi([X, Y])(p)
\end{eqnarray}
for each $p \in M$, where $Y$ is a $C^r$-vector field on a neighborhood of $p$ in $M$ which projects to  $\bar Y$. Since $TM/D$ has constant rank, then the projections of any local generators for $TM$ span the module of local section of $TM/D$. Hence the existence of  $Y$ is ensured.
\begin{proposition}\label{pro:2.12}
The map $\nabla$ defined by \eqref{eq:2.1} is well-defined and satisfies the following conditions:\\
1. $(X, \bar Y)$ is  $\mathbb{R}$-bilinear.\\
2. $\nabla _{fX} \bar Y = f\nabla _X \bar Y$ for $f \in C^r  (M)$.\\
3. $\nabla _X (f\bar Y) = f\nabla _X \bar Y + (L_X f)\bar Y$ for $f \in C^r(M)$.
\end{proposition}
\begin{proof}
Let $p \in M$ and $Y_1, Y_2$ be two  $C^r$-sections defined on a local neighborhood $U$ of $p$ such that $\bar Y|U=\pi(Y_1)=\pi(Y_2)$. It follows that
\begin{equation}
[X, Y_1-Y_2] \in [X, \Gamma^r(U;D)] \subseteq \Gamma^r(U; D),\nonumber
\end{equation}
for $X\in \Gamma^r(U;D)$. That is,
\begin{equation}\label{eq:2.2}
\pi([X,Y_1])=\pi([X,Y_2]).
\end{equation}
Hence $\nabla$ is well-defined. The proof for  the three conditions is obvious, so we omit it. \qed
\end{proof}

Then  from \cite{Bott,Iliev} we know that the map $\nabla$ is a connection on $TM/D$.  Let $c:[0,1] \rightarrow M$ be a $C^{1}$ curve such that $\dot c(t) \in D_{c(t)}$ for each $t \in [0,1]$. Recall that $\sigma: [0,1] \rightarrow TM/D$ is said to be a parallel translation along $c$, if $\nabla_{\dot c(t)}\sigma(t)=0$ for all $t \in [0,1]$. Since $\nabla$ satisfies that $\nabla_{X_1}\nabla_{X_2}\bar Y-\nabla_{X_2}\nabla_{X_1}\bar Y-\nabla_{[X_1, X_2]}\bar Y=\pi([X_1, [X_2, Y]]-[X_2,[X_1, Y]]-[[X_1, X_2], Y])=0$, $\nabla$ is flat. Then from \cite{Bott,Iliev} we know that for each $ p \in M$, there exists a neighborhood on the leaf of foliation of $D$ through $p$, such that when restricted to this neighborhood, the parallel translation depends only on $c(0)$ and $c(1)$.

Let $G$ be a $C^r$-generalized distribution of $M$. From local coordinates representation of $\pi$, we know  that $\pi(G+D)$ is a $C^r$-generalized subbundle of $TM/D$. Let $U$ be a local neighborhood of $p$ and $\phi: U \rightarrow \mathbb{R}^n$ be a local chart about $p$ satisfying that $\phi(p)=0$ and $D=\text{span}_{\mathbb{R}}\{\frac{\partial}{\partial {x_1}}, \cdots, \frac{\partial}{\partial {x_k}}\}$. Let $g_1, \cdots, g_m$ be local generators for $G$ on  a neighborhood $U$ of $p$.
We define the sections $\bar Z_i: U \rightarrow TM/D$ as follows:

For any $q \in U$ with $\phi(q)=(q_1,  \cdots, q_n)$, consider  the parallel translation
\begin{equation}\label{eq:2.7}
\sigma_i: [0,1] \rightarrow  TM/D
\end{equation}
of
\begin{equation}\label{eq:2.8}
\sigma_i(0)=\pi(g_i(\phi^{-1}(0, \cdots, 0, q_{k+1}, \cdots, q_n)))
\end{equation}
along any curve $c: [0,1] \rightarrow U$ with $\dot c(t) \in D_{c(t)}$ such that
\begin{equation}\label{eq:2.9}
\phi(c(0))=(0,\cdots, 0, q_{k+1}, \cdots, q_n)
\end{equation}
and $c(1)=q$. We define  $\bar Z_i(q)=\sigma_i(1), i=1, \cdots, m$.

Let $h_i$ be defined by
\begin{equation}\label{eq:3.2}
h_i(q)=\sum_{j=k+1}^{n}g_i^{j}(0, \cdots, 0, q_{k+1}, \cdots, q_n)\frac{\partial}{\partial x^{j}},
\end{equation}
for $q \in U,  i=1, \cdots, m$. Direct computation yields that $\bar Z_i(q)=\pi(h_i(q)),  i=1, \cdots, m$.  Hence $\bar Z_i$ are $C^r$-sections of $TM/D, i=1, \cdots,m$.
\begin{proposition}\label{pro:3.3}
Assume that $g_1, \cdots, g_m$ satisfy that
\begin{equation}\label{eq:3.7}
[g_i,\Gamma^r(U; D)]\subseteq \Gamma^r(U;D)+\text{span}_{C^r(U)}\{g_1, \cdots, g_m\}, i=1, \cdots, m.
\end{equation}
Then we have \[\pi(g_i) \in \text{span}_{C^{r}(U)}\{\bar Z_1, \cdots, \bar Z_m\}, i=1, \cdots, m,\]
where $\text{span}_{C^r(U)}\{\bar Z_1, \cdots, \bar Z_m\}$ denotes the $C^r(U)$-module spanned by $\bar Z_1, \cdots, \bar Z_m$.
\end{proposition}
\begin{proof}
From \eqref{eq:3.7} we know that there exist functions $\gamma^l_{ij} \in C^r(U)$ such that
\begin{equation}\label{eq:3.8}
\nabla_{\frac{\partial}{\partial x_i}}{\pi(g_j)}=\sum_{l=1}^m \gamma^l_{ij}\pi(g_l), i=1,\cdots,k, j=1, \cdots,m.
\end{equation}
Let $\hat \zeta_{ii}(0,\cdots,0,q_{k+1},\cdots,q_n)=1$, $\hat \zeta_{ij}(0,\cdots,0,q_{k+1},\cdots,q_n)=0, j\neq i, i,j=1,\cdots,m.$ Then it follows from the definition of $\bar Z_j$ that
\begin{align*}
&\sum_{j=1}^m \hat \zeta_{ij}(0,\cdots,0,q_{k+1},\cdots,q_n)\bar Z_j(0,\cdots,0,q_{k+1},\cdots,q_n)\\
=&\pi(g_i)(0,\cdots,0,q_{k+1},\cdots,q_n).
\end{align*}
Assume that there exist $\hat \eta_{ij}(q_1, \cdots, q_l,0,\cdots,0,q_{k+1}, \cdots, q_n)$ satisfying that
\begin{eqnarray}\label{eq:3.3}
\sum_{j=1}^m &&\hat \eta_{ij}(q_1, \cdots, q_l,0,\cdots,0,q_{k+1}, \cdots, q_n)\bar Z_j(q_1, \cdots, q_l,0,\cdots,0,q_{k+1}, \cdots, \nonumber\\
&&q_n)=\pi(g_i)(q_1, \cdots, q_l,0,\cdots,0,q_{k+1}, \cdots, q_n).
\end{eqnarray}
We claim there exist $\hat \xi_{ij}(q_1, \cdots, q_l,q_{l+1},0,\cdots,0,q_{k+1}, \cdots, q_n)$ satisfying that
\begin{align*}
\sum_{j=1}^m &\hat \xi_{ij}(q_1, \cdots, q_l,q_{l+1},0,\cdots,0,q_{k+1}, \cdots, q_n)\bar Z_j(q_1, \cdots, q_l,q_{l+1},0,\cdots,0, \nonumber\\
&q_{k+1}, \cdots, q_n)=\pi(g_i)(q_1, \cdots, q_l,q_{l+1}, 0,\cdots,0,q_{k+1}, \cdots, q_n).
\end{align*}
Consider the curve $c_{(q_1,\cdots, q_l,q_{k+1},\cdots, q_n)}:[-1,1] \rightarrow M$ with
\begin{equation}\label{eq:27}
\phi(c_{(q_1,\cdots, q_l,q_{k+1},\cdots, q_n)}(0))=(q_1,\cdots,q_l,0, \cdots, 0, q_{k+1},\cdots,q_n)
\end{equation}
and
\begin{equation}\label{eq:26}
\dot c_{(q_1,\cdots, q_l,q_{k+1},\cdots, q_n)}(t)=\frac{\partial}{\partial x_{l+1}}.
\end{equation}
$c_{(q_1,\cdots, q_l,q_{k+1},\cdots, q_n)}(t)$ lies in the leaf of $\rho(K)$ passing $c_{(q_1,\cdots, q_l,q_{k+1},\cdots, q_n)}(0)$.

Let $\bar Y_j=\sum_{i=1}^m\xi_{ji}\bar Z_i, j=1, \cdots, m$. We have
\begin{eqnarray}
&&\nabla_{\dot {c}_{(q_1,\cdots, q_l,q_{k+1},\cdots, q_n)}(t)}\bar Y_j=\sum_{r=1}^m\dot{\xi}_{jr}(c_{(q_1,\cdots, q_l,q_{k+1},\cdots, q_n)}(t))\bar Z_r\nonumber\\
&&+\sum_{r=1}^m{\xi}_{jr}( {c}_{(q_1,\cdots, q_l,q_{k+1},\cdots, q_n)}(t))\nabla_{\dot {c}_{(q_1,\cdots, q_l,q_{k+1},\cdots, q_n)}(t)} \bar Z_r\nonumber\\
&&=\sum_{r=1}^m\dot{\xi}_{jr}(c_{(q_1,\cdots, q_l,q_{k+1},\cdots, q_n)}(t))\bar Z_r(c_{(q_1,\cdots, q_l,q_{k+1},\cdots, q_n)}(t)).
\end{eqnarray}
On the other hand,
\begin{equation}
\sum_{r=1}^m \gamma^r_{{l+1}j}\bar Y_r=\sum_{r=1}^m \gamma^r_{{l+1}j}\sum_{a=1}^m\xi_{ra}\bar Z_a=\sum_{r=1}^m\sum_{a=1}^m\gamma^r_{{l+1}j}\xi_{ra}\bar Z_a.
\end{equation}
Set \[\dot{\sigma}_{jr}(t)=\sum_{a=1}^m\gamma^a_{{l+1}j}(c_{(q_1,\cdots, q_l,q_{k+1},\cdots, q_n)}(t))\sigma_{ar}(t), j, r=1,\cdots,m.\]
Then we get $m$-system of linear ordinary differential equations with parameters $q_1,\cdots, q_l,q_{k+1},\cdots, q_n$ which has
unique solutions with prescribed initial value
\begin{equation}\label{eq:3.5}
\sigma_{jr}(0,q_1,\cdots, q_l,q_{k+1},\cdots, q_n)=\hat \eta_{jr}(q_1,\cdots, q_l, 0, \cdots, 0, q_{k+1},\cdots, q_n),
\end{equation}
for $\hat \eta_{jr}$ given in \eqref{eq:3.3} and $j,r=1, \cdots, m$.
It follows that
\begin{equation}\label{eq:3.4}
\nabla_{\dot {c}_{(q_1,\cdots, q_l,q_{k+1},\cdots, q_n)}(t)}{\sum_{r=1}^m \sigma_{jr}(t)\bar Z_r}=\sum_{r=1}^m \gamma^r_{{l+1}j}(c_{(q_1,\cdots, q_l,q_{k+1},\cdots, q_n)}(t))\sum_{a=1}^m \sigma_{ra}(t)\bar Z_a,
\end{equation}
for $j=1, \cdots, m$.

Since there exists unique solution for system of equations \[\nabla_{\dot {c}_{(q_1,\cdots, q_l,q_{k+1},\cdots, q_n)}(t)}{Y_i}=\sum_{r=1}^m \gamma^r_{{l+1}i}Y_r, i=1,\cdots,m,\] with prescribed initial value, it follows from \eqref{eq:3.8}, \eqref{eq:3.3}, \eqref{eq:3.5} and \eqref{eq:3.4} that
\begin{align*}
\sum_{r=1}^m &\sigma_{jr}(q_{l}, q_1,\cdots, q_{l+1},  q_{k+1}, \cdots, q_n)\bar Z_r(q_1,\cdots, q_{l+1}, 0, \cdots, 0, q_{k+1}, \\
&\cdots, q_n)=\pi(g_j)(q_1,\cdots, q_{l+1}, 0, \cdots, 0, q_{k+1}, \cdots, q_n).
\end{align*}
Let $\hat \xi_{ij}(q_1, \cdots, q_l,q_{l+1},0,\cdots,0,q_{k+1}, \cdots, q_n)=\sigma_{jr}(q_{l+1}, q_1,\cdots, q_{l},  q_{k+1}, \cdots, q_n)$. This completes the proof for the above claim. By induction, we know that
there exist $\hat \lambda_{ij} \in C^{r} (U), i,j =1, \cdots,m,$ such that
\begin{equation}
\sum_{j=1}^m \hat  \lambda_{ij}\bar Z_j=\pi(g_i),
\end{equation}
for $i=1, \cdots, m$. So the result follows immediately.\qed
\end{proof}

Proof of Theorem \ref{tm:3.5}.
The proof of the "only if" part is obvious according to Definition \ref{def:1.1} and \ref{def:1.2}. We only need to prove the "if" part.  It follows from \eqref{eq:3.9a} and Proposition \ref{pro:3.3}  that
\begin{equation}\label{eq:2.13}
(\pi(g_1) \cdots \pi(g_m))=(\bar Z_1 \cdots \bar Z_m)A,
\end{equation}
where $A=(A_{ij})_{m \times m}$ satisfies that $A_{ij}\in C^r(U), i,j=1, \cdots, m$. From the definition of $\bar Z_j$ we know that $A(p)=\text{Id}_{m \times m}$.
Then by shrinking $U$ to $V \subseteq U$ if necessary, we can ensure that  $A(q) \in GL(m, \mathbb{R})$ for each $q\in V$. Consider the vector fields $\hat g_1, \cdots, \hat g_m$ on $V$ defined by
\begin{equation}\label{eq:2.22}
(\hat g_1 \cdots \hat g_m)=(g_1 \cdots g_m)A^{-1}.
\end{equation}
It follows from \eqref{eq:2.13} that $\pi(\hat g_j)=\bar Z_j, j=1, \cdots, m$, which yield that
\[\pi([X,  \hat{g}_j])=\nabla_{X}\pi(\hat g_j)=\nabla_{X}\bar Z_j=0, j=1, \cdots, m,\]
for any $X \in \Gamma^r(V; D)$.
That is,
\begin{equation}\label{eq:2.23}
[\hat{g}_j, \Gamma^r(V;D)] \subseteq \Gamma^r(V;D),  j=1, \cdots, m.
\end{equation}
Hence \eqref{eq:1.3b} follows immediately from \eqref{eq:2.22} and \eqref{eq:2.23}, where $\beta=A^{-1}$.

It  follows from \eqref{eq:3.9b}  that
\[\nabla_{X}\pi(f) \in \text{span}_{C^{r}(U)}\{\pi(g_1), \cdots, \pi(g_m)\}\subseteq\text{span}_{C^{r}(U)}\{\bar Z_1, \cdots, \bar Z_m\},\]
for any $X \in \Gamma^r(U; D)$.
Hence there exist  $\alpha_{ij} \in C^r(U)$, such that
\begin{equation}\label{eq:2.14}
\nabla_{\frac{\partial}{\partial x_i}}\pi(f)=\sum_{j=1}^m \alpha_{ij}\bar Z_j,
\end{equation}
for $i=1, \cdots, k,  j=1, \cdots, m$.
We claim that there exists $\bar g=\sum_{j=1}^m \beta_j\bar Z_j$, such that
\begin{equation}\label{eq:2.15}
\nabla_{\frac{\partial}{\partial x_i}}\bar g=\sum_{j=1}^m \alpha_{ij}\bar Z_j, i=1, \cdots, k.
\end{equation}
Let
\begin{eqnarray}
\beta_j&=&\int_0^{x_k}\alpha_{kj}(x_1, \cdots, x_{k-1}, \tau, x_{k+1}, \cdots, x_n)d\tau+\int_0^{x_{k-1}}\alpha_{(k-1)j}(x_1,\cdots, x_{k-2},\nonumber\\
&&\tau,0, x_{k+1}, \cdots, x_n)d\tau+\cdots+\int_0^{x_1}\alpha_{1j}(\tau, 0, \cdots, 0, x_{k+1}, \cdots, x_n)d\tau,
\end{eqnarray}
for $j=1, \cdots, m$.

Since $\nabla$ satisfies that $\nabla_{\frac{\partial}{\partial x_i}}\nabla_{\frac{\partial}{\partial x_j}}\pi(f)-\nabla_{\frac{\partial}{\partial x_j}}\nabla_{\frac{\partial}{\partial x_i}}\pi(f)-\nabla_{[\frac{\partial}{\partial x_i}, \frac{\partial}{\partial x_j}]}\pi(f)=0$, we have
\begin{equation}\label{eq:2.19}
\sum_{l=1}^m\frac{\partial \alpha_{il}}{\partial x_j}\bar Z_l=\sum_{l=1}^m\frac{\partial \alpha_{jl}}{\partial x_i}\bar Z_l, i,j=1, \cdots, k.
\end{equation}
Hence for $i=1, \cdots, k$,
\begin{eqnarray}\label{eq:2.17}
&&\nabla_{\frac{\partial}{\partial x_i}}\bar  g=\nabla_{\frac{\partial}{\partial x_i}}\sum_{j=1}^m \beta_j\bar Z_j=\sum_{j=1}^m {\frac{\partial}{\partial x_i}}(\beta_j)\bar Z_j\nonumber\\
&=&\sum_{j=1}^m(\int_0^{x_k}{\frac{\partial}{\partial x_i}} \alpha_{kj}(x_1, \cdots, x_{k-1}, \tau, x_{k+1}, \cdots, x_n)d\tau+\int_0^{x_{k-1}}{\frac{\partial} {\partial {x_i}}} \alpha_{(k-1)j}\nonumber\\
&&(x_1, \cdots, x_{k-2},\tau,0, x_{k+1}, \cdots, x_n)d\tau+\cdots+\nonumber\\
&&\alpha_{ij}(x_1, \cdots, x_i, 0, \cdots, 0,, x_{k+1}, x_n))\bar Z_j.
\end{eqnarray}
From \eqref{eq:3.2} we know that in the local chart $(U, \phi)$, the section $\bar Z_j$ only depends on  $x_{k+1}, \cdots, x_n$, for $j=1, \cdots, m$. It follows that
\begin{eqnarray}
(35)&=&\int_0^{x_k}\sum_{j=1}^m{\frac{\partial \alpha_{kj}}{\partial x_i}}\bar Z_j (x_1, \cdots, x_{k-1}, \tau, x_{k+1}, \cdots, x_n)d\tau+\int_0^{x_{k-1}}\sum_{j=1}^m{\frac{\partial \alpha_{(k-1)j}}{\partial x_i}}\bar Z_j\nonumber\\
&&(x_1, \cdots, x_{k-2},\tau,0, x_{k+1}, \cdots, x_n)d\tau+\cdots+\sum_{j=1}^m\alpha_{ij}(x_1, \cdots, x_i, 0, \cdots, 0,, x_{k+1}, x_n)\bar Z_j\nonumber\\
&=&\int_0^{x_k}\sum_{j=1}^m{\frac{\partial \alpha_{ij}}{\partial x_k}}\bar Z_j (x_1, \cdots, x_{k-1}, \tau, x_{k+1}, \cdots, x_n)d\tau+\int_0^{x_{k-1}}\sum_{j=1}^m{\frac{\partial \alpha_{ij}}{\partial x_{k-1}}}\bar Z_j\nonumber\\
&&(x_1, \cdots, x_{k-2},\tau,0, x_{k+1}, \cdots, x_n)d\tau+\cdots+\sum_{j=1}^m\alpha_{ij}(x_1, \cdots, x_i, 0, \cdots, 0,, x_{k+1}, x_n)\bar Z_j\nonumber\\
&=&\sum_{j=1}^m(\int_0^{x_k}{\frac{\partial}{\partial \tau}} \alpha_{ij}(x_1, \cdots, x_{k-1}, \tau, x_{k+1}, \cdots, x_n)d\tau+\int_0^{x_{k-1}}{\frac{\partial} {\partial \tau}} \alpha_{ij}\nonumber\\
&&(x_1, \cdots, x_{k-2},\tau,0, x_{k+1}, \cdots, x_n)d\tau+\cdots+\alpha_{ij}(x_1, \cdots, x_i, 0, \cdots, 0,, x_{k+1}, x_n))\bar Z_j\nonumber\\
&=&\sum_{j=1}^m (\alpha_{ij}(x_1, \cdots, x_{k-1}, x_k, x_{k+1}, \cdots, x_n)-\alpha_{ij}(x_1, \cdots, x_{k-1}, 0, x_{k+1}, \cdots, x_n)\nonumber\\
&+&\alpha_{ij}(x_1, \cdots, x_{k-1}, 0, x_{k+1}, \cdots, x_n)-\alpha_{ij}(x_1, \cdots, x_{k-2}, 0, 0, x_{k+1}, \cdots, x_n)+\nonumber\\
&&\cdots+\alpha_{ij}(x_1, \cdots, x_i, 0, \cdots, 0,, x_{k+1}, x_n))\bar Z_j\nonumber\\
&=&\sum_{j=1}^m \alpha_{ij}(x_1, \cdots, x_{k-1}, x_k, x_{k+1}, \cdots, x_n)\bar Z_j,
\end{eqnarray}
where the second equality follows from \eqref{eq:2.19}. This completes the proof for the above claim.
From \eqref{eq:2.13} we know that, by shrinking $U$ if necessary, there exist $\hat \alpha_1, \cdots, \hat \alpha_m \in C^r(U)$, such that
$\bar g=\sum_{j=1}^m \hat \alpha_j\pi(g_j)$.

Consider the vector field $\hat f=f-\sum_{j=1}^m \hat \alpha_j g_j$. It follows from 1 and 2 in Proposition \ref{pro:2.12} that
\begin{equation}\label{eq:2.21}
\nabla_{X}(\pi(\hat f))=\nabla_{X}(\pi(f)-\bar g)=0,
\end{equation}
for any $X \in \Gamma^r(U; D)$. Hence $\hat f$ satisfies that
\[[\hat f, \Gamma^r(U; D)] \subseteq \Gamma^r(U; D).\]
Hence \eqref{eq:1.3a} follows immediately, where $\alpha=(-\hat \alpha_1 \cdots -\hat \alpha_m)$. This completes the proof for Theorem \ref{tm:3.5}.
\qed
\begin{remark}\label{remark 1}
In \cite{Cheng} Cheng and Tarn give the following sufficient and necessary conditions for local controlled invariance of the smooth involutive regular distribution $D$ for the smooth affine control system $\Sigma$:
\begin{subequations}
\begin{align}
&[f, D]/D\subseteq \text{span}_{C^{\infty}(U)}\{g_j/D|j=1, \cdots, m\},\label{eq:3.26a}\\
&[g_i, D]/D\subseteq \text{span}_{C^{\infty}(U)}\{g_j/D|j=1, \cdots, m\},i=1, \label{eq:3.26b}\cdots, m.
\end{align}
\end{subequations}
\eqref{eq:3.26a} and \eqref{eq:3.26b} are equivalent with our result \eqref{eq:3.9a} and \eqref{eq:3.9b}. However, the proof provided by \cite{Cheng} is not a geometric one. Here we provide a differential geometric proof for this result.
\end{remark}
\begin{corollary}\label{coro:3.6}
Consider the $C^r$-affine control system \eqref{eq:1.1} on $M$.  Let $D$ be an involutive regular $C^r$-distribution of $M$. Let $p$ be a regular point for the distribution $G$. Then $D$ is locally controlled invariant for $\Sigma$ at $p$, if and only if $D$ is locally controlled invariant for $\Delta$ at $p$, where $\Delta$ is the $C^r$-generalized affine distribution defined by $\Delta(p)=f(p)+\text{span}_{\mathbb{R}}\{g_1(p), \cdots, g_m(p)\}$, for each $p\in M$, if and only if there exists a sufficiently small  neighborhood $U$ of $p$, such that
\begin{equation}\label{eq:3.21}
[\Gamma^r(U;\Delta),\Gamma^r(U;D)]\subseteq \Gamma^r(U;D)+\Gamma^r(U;G).
\end{equation}
\end{corollary}
\begin{proof}
Since $p$ is a regular point for the distribution
$G$, if $U$ is a sufficiently small neighborhood of $p$,
$\Gamma^{r}(U; G)=\text{span}_{C^r(U)}\{g_1, \cdots, g_m\}$. Assume \eqref{eq:3.21} holds, then we have
\[[g_i,\Gamma^r(U;D)]\subseteq \Gamma^r(U;D)+\text{span}_{C^r(U)}\{g_1, \cdots, g_m\},\]
and
\[[f,\Gamma^r(U;D)]\subseteq \Gamma^r(U;D)+\text{span}_{C^r(U)}\{g_1, \cdots, g_m\}.\]
It follows from Theorem \ref{tm:3.5} that  $D$ is locally controlled invariant for $\Sigma$ at $p$.
If $D$ is locally controlled invariant for $\Sigma$ at $p$, it is obvious that $D$ is locally controlled invariant for $\Delta$ at $p$. Assume that  $D$ is locally controlled invariant for $\Delta$ at $p$, then there exist generators $\hat g_i \in \Gamma^{r}(U;G), i\in I$ of  $G$, such that $[\hat g_i, \Gamma^r(U;D)] \subseteq \Gamma^r(U;D)$. Since $p$ is a regular point, let $U$ be sufficiently small such that $\Gamma^{r}(U; G)=\text{span}_{C^r(U)}\{\hat g_{i_1}, \cdots, \hat g_{i_k}\}$, where $i_1, \cdots, i_k\in I$.  Then \eqref{eq:3.21} follows immediately by direct calculation. This completes the proof.\qed
\end{proof}
\begin{corollary}\label{coro:3.7}
Consider the $C^r$-affine control system \eqref{eq:1.1} on $M$.  Let $D$ be an involutive regular $C^r$-distribution of $M$. Let $p$ be a regular point for the generalized subbundle
$(G+D)/D$. Then $D$ is locally controlled invariant for $\Sigma$ at $p$, if and only if there exists a sufficiently small neighborhood $U$ of $p$, such that
\begin{subequations}
\begin{align}
&[g_i,\Gamma^r(U; D)]\subseteq \Gamma^r(U;D)+\Gamma^r(U;G), i=1, \cdots, m,\label{eq:3.22a}\\
&[f, \Gamma^r(U;D)] \subseteq \Gamma^r(U;D)+\Gamma^r(U;G).\label{eq:3.22b}
\end{align}
\end{subequations}
\end{corollary}
\begin{proof}
To prove the "if" part, since $p$ is a regular point for the generalized subbundle
$(G+D)/D$,  we have
\begin{equation}\label{eq:3.23}
\Gamma^r(U; \pi(G))=\text{span}_{C^r(U)}\{\pi(g_1), \cdots, \pi(g_m)\},
\end{equation}
if $U$ is sufficiently small. It follows from \eqref{eq:3.22b} that
$\pi([f, X])\in \Gamma^r(U; \pi(G))=\text{span}_{C^r(U)}\{\pi(g_1), \cdots, \pi(g_m)\}$, for any $X \in  \Gamma^r(U; D)$.
Hence \eqref{eq:3.9b} holds. Similar arguments can be applied to $g_i$. Hence \eqref{eq:3.9a} holds. Then the "if" part follows  from Theorem \ref{tm:3.5}. The proof for the "only if" part is obvious.\qed
\end{proof}
\begin{remark}
As a classical result on controlled invariance, Corollary \ref{coro:3.7} is presented in the classical books \cite{schaft,Isidory} on geometric control theory. It also appears as Lemma 1.5  in \cite{1}. As mentioned in Section \ref{intro}, the geometric interpretation for Corollary \ref{coro:3.7} provided by  Lemma 3.1 in \cite{1} is not a complete one. The detailed reason is given as follows: compared with our definition of sections $\bar Z_i$ on the open subset $U$ of $M$ given by \eqref{eq:2.7}, \eqref{eq:2.8} and \eqref{eq:2.9},  the parallel vectors $\bar Z_i$ given  in  Lemma 3.1 in \cite{1} are defined only on  each leaf of the foliation,  not on any open subset of $M$. As a consequence, consider the local chart $\phi: U\rightarrow \mathbb{R}^n$ such that $D=\text{span}_{\mathbb{R}}\{\frac{\partial}{\partial x_1}, \cdots, \frac{\partial}{\partial x_k}\}$, and the projection map $\psi(x_1, \cdots, x_n)=(x_{k+1}, \cdots, x_n)$. Then Lemma 3.1 is equivalent with the following statement: $T_{x}\psi(G)=T_y\psi(G)$ if $\psi(x)=\psi(y)$. But neither the smoothness of the distribution $T\psi(G)$ of the orbit space $\psi(U)$, nor the existence of $C^r$-feedback functions $(\alpha, \beta)$ and the corresponding vector fields on any open subset of $M$ satisfying \eqref{eq:1.3a} and \eqref{eq:1.3b},
are studied  in  Lemma 3.1. In fact,  Lemma 3.1 is weaker than Lemma 1.5 in \cite{1}.

In the following we present a direct geometric proof for Corollary \ref{coro:3.7} as well as Lemma 1.5 in \cite{1}, which improves the proof given in \cite{1} and is simpler than the proof for the general case presented by Proposition \ref{pro:3.3}.
\begin{proof}
Consider the curve  $c: [0,1] \rightarrow U$ with $c(0)$ given by \eqref{eq:2.9}, $c(1)=q$ and  $\dot c(t) \in D_{c(t)}$.
Let $\sigma_i(t)=\sum_{j=1}^m {f}_{ij}(c(t))\pi({g}_j(c(t)))$, $i=1, \cdots, m$, where $f_{ij} \in C^r(U)$ satisfy that
$f_{ii}(0, 0, \cdots, q_{k+1}, \cdots, q_n)=1$, $f_{ij}(0, \cdots, 0,  q_{k+1}, \cdots, q_n)=0$, for $j \neq i$.

According to \eqref{eq:3.22a} and \eqref{eq:3.23},  there exist $\gamma_{ij} \in C^r(U)$, $j=1, \cdots, m$, such that
\[\nabla_{\dot c(t)}\pi(g_i(c(t))=\sum_{j=1}^m \gamma_{ij}(c(t))\pi(g_j(c(t)).\]
On the other hand,  it follows from  \eqref{eq:3.22a} and \eqref{eq:3.23} that $\nabla_{\frac{\partial}{\partial x^{l}}} \pi({g}_j)=\sum_{r=1}^m\Gamma^r_{lj}\pi(g_r),$
where $\Gamma^k_{lj} \in C^r(U), l=1,\cdots, k,  j=1, \cdots, m$. Then we have
\begin{eqnarray}
&&\nabla_{\dot c(t)}\sigma_i(t)=\nabla_{\dot c(t)}\sum_{j=1}^m {f}_{ij}(c(t))\pi({g}_j(c(t)))
=\sum_{j=1}^m \nabla_{\dot c(t)} {f}_{ij}(c(t))\pi({g}_j(c(t)))\nonumber\\
&=&\sum_{j=1}^m \dot f_{ij}(c(t))\pi({g}_j(c(t)))+ {f}_{ij}(c(t))\nabla_{\sum_{l=1}^{k}c^{l}(t) \frac{\partial}{\partial x^{l}}} \pi({g}_j(c(t)))\nonumber\\
&=&\sum_{j=1}^m \dot f_{ij}(c(t))\pi({g}_j(c(t)))+  {f}_{ij}(c(t))\sum_{l=1}^{k}c^{l}(t)\sum_{r=1}^{m}
\Gamma_{lj}^r(c(t)))\pi(g_r(c(t)))\nonumber\\
&=&\sum_{j=1}^m(\dot f_{ij}(c(t))+\sum_{l=1}^k\sum_{r=1}^{m}f_{ir}(c(t))c^l(t)\Gamma^{j}_{lr}(c(t)))\pi({g}_j(c(t))).
\end{eqnarray}

Let $\dot f_{ij}(c(t))+\sum_{l=1}^k\sum_{r=1}^{m}f_{ir}(c(t))c^l(t)\Gamma^{j}_{lr}(c(t))=0$, for $j=1, \cdots, m$. Then we get $m$ systems of linear equations
\[\dot f_i=A_if_i, i=1, \cdots, m,\]
with $f_i=(f_{i1}, \cdots, f_{im})^{T}, A_i=(a_i^{jl})_{m\times m}$ where $a_i^{jl}(t)=\sum_{r=1}^k c^r(t)\Gamma^{j}_{rl}(c(t))$. We know there exist unique solutions $\hat f_i(t)$ defined on $[0,1]$ with prescribed initial value $\hat f_{ii}(0)=1, \hat f_{ij}(0)=0$, for $j \neq i, i=1, \cdots,m$.

Let  $\hat \sigma_i(t)=\sum_{j=1}^m \hat {f}_{ij}(c(t))\pi({g}_j(c(t)))$.  We have $\nabla_{\dot c(t)}\hat \sigma_i(t)=0$. It follows from uniqueness of the solution of the equation $\nabla_{\dot c(t)}\sigma_i(t)=0$ with prescribed initial value that $\hat \sigma_i(t)$ is the parallel translation along $c(t)$ with  initial value defined by \eqref{eq:2.8}. Then according to our previous definition of $\bar Z_i$, we have
\begin{equation}
\bar Z_i(q)=\hat \sigma_i(1)=\sum_{j=1}^m \hat {f}_{ij}(c(1))\pi({g}_j(c(1))) \in \pi(G)_q,
\end{equation}
for $i=1, \cdots, m$. Since $p$ is a regular point for $(G+D)/D$, by shrinking $U$ if necessary we know there exist functions $\xi_{ij}\in C^r(U)$,
such that $\bar Z_i=\sum_{j=1}^m\xi_{ij}\pi(g_j), i,j=1,\cdots,m$. The remaining proof follows  the proof for Theorem \ref{tm:3.5}. \qed
\end{proof}
\end{remark}

\begin{corollary}\label{coro:3.9}
Let $M$ be a $C^{\omega}$-manifold. Let $\Delta$ be a $C^{\omega}$-generalized  affine distribution of $M$ with associated generalized distribution $G=L(\Delta)$. Let $D$ be an involutive  regular $C^{\omega}$-distribution of $M$. Then $D$ is locally controlled invariant for $\Delta$ at $p$, if there exists a sufficiently small  neighborhood $U$ of $p$, such that
\begin{equation}\label{eq:3.24}
[\Gamma^{\omega}(U;\Delta),\Gamma^{\omega}(U;D)]\subseteq \Gamma^{\omega}(U;D)+\Gamma^{\omega}(U;G).
\end{equation}
\end{corollary}
\begin{proof}
Since $G$ is of class $C^{\omega}$, if $U$ is sufficiently small, then it follows from  Proposition \ref{pro:2.13} that there exist local generators $g_1, \cdots, g_m \in \Gamma^{\omega}(U; G)$ of $G$ such that
\begin{equation}\label{eq:3.25}
\Gamma^{\omega}(U; G)=\text{span}_{C^{\omega}(U)}\{g_1, \cdots, g_m\}.
\end{equation}
Then we have
\[[g_i,\Gamma^{\omega}(U;D)]\subseteq \Gamma^{\omega}(U;D)+\text{span}_{C^{\omega}(U)}\{g_1, \cdots, g_m\},\]
and
\[[f,\Gamma^{\omega}(U;D)]\subseteq \Gamma^{\omega}(U;D)+\text{span}_{C^{\omega}(U)}\{g_1, \cdots, g_m\}.\]
Then the result follows immediately from Theorem \ref{tm:3.5}.\qed
\end{proof}

\section{Global controlled invariance}
\label{sec:4}
\subsection{General cases}
Let $D$ be an involutive regular $C^r$-distribution of $M$. Consider the foliation $\mathfrak{F}$ defined by $D$.  As claimed  in \cite{1},  $M/\mathfrak{F}$ is assumed to be a $C^r$-manifold in this subsection.
\begin{theorem}\label{tm:4.1}
 Let $M$ be a smooth manifold and $D$ be a smooth involutive regular distribution of $M$ such that $M/\mathfrak{F}$ is a smooth manifold, where $\mathfrak{F}$ is the foliation defined by $D$. Let $\Delta$ be a smooth generalized affine distribution of $M$ with associated smooth generalized distribution $L(\Delta)=G$.  Assume that $(G+D)/D$ is  a regular subbundle of $TM/D$. Assume also that $G \cap D$ is  a smooth generalized distribution of $M$. Then $D$ is globally controlled invariant for $\Delta$ if and only if it is locally controlled invariant for $\Delta$ at each point $p \in M$.
\end{theorem}
\begin{proof}
The "only if " part follows from the arguments given in Section \ref{intro}. We only need to prove the "if" part. Consider the map $\Phi: M\rightarrow M/\mathfrak{F}$. Since $D$ is locally controlled invariant for $\Delta$, it follows that for each $p \in M$, there exists a local neighborhood $U$ of $p$, such that $T\Phi(G|U)$ is a well-defined smooth generalized distribution of $\Phi(U)$. Further, since $\Phi^{-1}(q)$ is a leaf of $\mathfrak{F}$, it is connected, which yields that  $T\Phi(G)$ is a well-defined smooth generalized distribution of $M/\mathfrak{F}$. Then it follows from Theorem \ref{tm:2.4} that there exist smooth vector fields $Y_1, \cdots, Y_l $ on  $M/\mathfrak{F}$ such that
\begin{equation}\label{eq:4.2}
(T\Phi(G))_{\bar p}=\text{span}_{\mathbb{R}}\{Y_1(\bar p), \cdots, Y_l(\bar p)\},
\end{equation}
for each $\bar p \in M/\mathfrak{F}$.

Given $Y \in \Gamma^{\infty}(T\Phi(G))$, we claim that there exists   $X \in \Gamma^{\infty}(G)$ such that $T\Phi(X)=Y$. Since $(G+D)/D$ is regular, it follows from \eqref{eq:3.23} in Corollary \ref{coro:3.7} that for each $p_{\alpha} \in M$, there exist a local neighborhood $U_\alpha$ of $p$, and $X_{\alpha} \in \Gamma^{\infty}(U_\alpha;G)$ such that $T\Phi(X_\alpha)=Y|\Phi(U_\alpha)$. Then we get an open covering
$\{U_\alpha\}$ of $M$. Since $M$ is a smooth, paracompact, Hausdorff manifold, let
$\{(\mathcal{W}_i, \lambda_i)\}$ be a partition of unity \cite{Abraham} subordinated to the open covering $\{U_\alpha\}$ such that
each open set $\mathcal{W}_i$ is a subset of some open set $\mathcal{U}_{\alpha(i)}$. Let $X$ be defined by
\[X(p)=\sum_{i}\lambda_{i}X_{\alpha(i)}(p)\]
a finite sum at each $p \in M$. Then $X \in \Gamma^{\infty}(G)$ satisfies that
\[T_{p}\Phi(X)=\sum_{i}\lambda_{i}(p)T_{p}\Phi(X_{\alpha(i)}(p)))=\sum_{i}\lambda_{i}(p)Y(\Phi(p))=Y(\Phi(p)),\]
for each $p \in M$. This completes the proof of the claim.

Hence we get $X_1, \cdots, X_l \in \Gamma^{\infty}(G)$ such that $T\Phi(X_i)=Y_i, i=1, \cdots, l$. On the other hand, since $G \cap D$ is  a smooth generalized distribution of $M$, it follows from Theorem \ref{tm:2.4} that there exist global generators $X_{l+1}, \cdots, X_r\in  \Gamma^{\infty}(G\cap D)$ for $G \cap D$.
We claim that $X_1, \cdots, X_r$ are global generators for $G$.
Let $v_p \in G_p$. From \eqref{eq:4.2} we know that  there exist constants $c_1, \cdots, c_l$ such that $T_p\Phi(v_p)=\sum_{i=1}^lc_iY_i(\Phi(p))$.
Then we have \[T\Phi(v_p-\sum_{i=1}^lc_iX_i(p))=T\Phi(v_p)-\sum_{i=1}^lc_iY_i(\Phi(p))=0.\] This yields that $v_p-\sum_{i=1}^lc_iX_i(p) \in D_p\cap G_p$. Then there exist constants $c_{l+1}, \cdots, c_r$  such that
\[v_p-\sum_{i=1}^lc_iX_i(p)=\sum_{i=l+1}^rc_iX_i(p).\]
Hence $v_p=\sum_{i=1}^rc_iX_i(p)$. That is, $X_1, \cdots, X_r$ are global generators for $G$.  It is obvious that $[X_i, \Gamma^{\infty}(D)] \subseteq \Gamma^{\infty}(D)$, for $i=1, \cdots, r$.

It now remains to prove the existence of $f\in \Gamma^{\infty}(\Delta)$ such that $[f, \Gamma^{\infty}(D)] \subseteq \Gamma^{\infty}(D)$. By using similar arguments as above, we know that the affine distribution $T\Phi(\Delta)$  of  $M/\mathfrak{F}$ is well-defined and smooth. By using partition of unity  we  get a smooth vector field $g$ on $M/\mathfrak{F}$ such that $g(\bar p) \in T\Phi(\Delta)_{\bar p}$ for each $\bar p \in M/\mathfrak{F}$. Since for each $p \in M$, there exist a neighborhood $U$ of $p$ and $f_p \in \Gamma^{\infty}(U; \Delta)$ such  that $T\Phi(f_p)-g|U \in \Gamma^{\infty}(\Phi(U); T\Phi(G))$, where $T\Phi(f_p) \in \Gamma^{\infty}(\Phi(U); T\Phi(\Delta))$. Then it follows from \eqref{eq:3.23} in Corollary \ref{coro:3.7} that,  by shrinking $U$ if necessary,  there exists $\hat f_p \in  \Gamma^{\infty}(U; \Delta)$ such that $T\Phi(\hat f_p)=g|\Phi(U)$. By using partition of unity as above, we then get a smooth vector field $\hat f=\sum_{i}\lambda_{i}\hat f_{p(i)}$ on $M$ such that $T\Phi(\hat f)=g$. This yields that $[\hat f, \Gamma^{\infty}(D)] \subseteq \Gamma^{\infty}(D)$. We only need to show that $\hat f(p)\in \Delta_p$ for each $p\in M$. Actually We have  a finite sum at $p$:
\begin{align*}
\hat f(p)&=\sum_{i}\lambda_{i}(p)\hat f_{p(i)}(p)=\sum_{i}\lambda_{i}(p)(v_0+v_{p_i})\\
&=(\sum_{i}\lambda_{i}(p))v_0+\sum_{i}\lambda_{i}(p)v_{p_i}=v_0+v \in \Delta_p,
\end{align*}
where $v_0 \in \Delta_p, v_{p_i}, v \in G_p$. This completes the proof.\qed
\end{proof}
\begin{remark}
It follows from the above proof that there exist finitely many global generators $X_1, \cdots, X_r \in \Gamma^{\infty}(G)$ of $G$, such that $[X_i, \Gamma^{\infty}(D)] \subseteq \Gamma^{\infty}(D)$, for $i=1, \cdots, r$. That is, in Theorem \ref{tm:4.1} the index set $I$ given in Definition \ref{def:1.4} can be finite.
\end{remark}
\begin{corollary}
 Let $M$ be a smooth manifold and $D$ be a smooth involutive regular distribution of $M$ such that $M/\mathfrak{F}$ is a smooth manifold, where $\mathfrak{F}$ is the foliation defined by $D$. Let $\Delta$ be a smooth affine distribution of $M$ with associated smooth distribution $L(\Delta)=G$. Assume that both $G$ and  $G \cap D$ have constant rank on $M$. Then $D$ is globally controlled invariant for $\Delta$ if and only if
\begin{equation}\label{eq:4.1}
[\Gamma^{\infty}(\Delta),\Gamma^{\infty}(D)]\subseteq \Gamma^{\infty}(D+G).
\end{equation}
\end{corollary}
\begin{proof}
Since both $G$ and  $G \cap D$ have constant rank on $M$, it follows that $(G+D)/D$ has constant rank and $G \cap D$ is smooth. Then according to Theorem \ref{tm:4.1}, $D$ is globally controlled invariant for $\Delta$ if and only if $D$ is locally controlled invariant for $\Delta$.
Since $G$ has constant rank, it follows from Corollary \ref{coro:3.6} that  $D$ is locally controlled invariant for $\Delta$ if and only if for each $p \in M$ there exists a local neighborhood $U$ of $p$, such that \eqref{eq:3.21} holds. Hence under the above assumptions, it suffices to show  the equivalence between
\eqref{eq:3.21} at every point in $M$ and  \eqref{eq:4.1}.

Assume that \eqref{eq:4.1} holds. Since $G$ and $D$ have constant rank, it follows that for each $p \in U$,
there exist  $g_1, \cdots, g_m$ and $d_1, \cdots, d_k \in \Gamma^{\infty}(U;D)$ which span the module $\Gamma^{\infty}(U;G)$ and $\Gamma^{\infty}(U;D)$ over $C^{\infty}(U)$ respectively. Besides, $g_1|V, \cdots, g_m|V$ and  $d_1|V, \cdots, d_k|V \in \Gamma^{\infty}(V;D)$ span the module $\Gamma^{\infty}(V;G)$ and $\Gamma^{\infty}(V;D)$ over $C^{\infty}(V)$  respectively, where $V$ is any open subset of $U$.
Then by using  partition of unity we get that, for each $p \in M$,  there exist generators for the module of local sections of $G$ and $D$ on a sufficiently small neighborhood $W$ of $p$ that are the restrictions of global sections of $G$ and $D$ to $W$ respectively. Then by restrictions of  globally defined sections in \eqref{eq:4.1} to $W$, we have
\[[\Gamma^{\infty}(W;G),\Gamma^{\infty}(W;D)]\subseteq \Gamma^{\infty}(W;D+G).\]
Since $G, D$ and $G\cap D$  have constant rank, it follows that $\Gamma^{\infty}(W; D+G)= \Gamma^{\infty}(W; D)+ \Gamma^{\infty}(W; G)$ if $W$ is sufficiently small. Hence we have
\[[\Gamma^{\infty}(W;G),\Gamma^{\infty}(W;D)]\subseteq \Gamma^{\infty}(W;D)+\Gamma^{\infty}(W;G).\]
Similarly we can prove that $[f, \Gamma^{\infty}(W;D)]\subseteq \Gamma^{\infty}(W;D)+\Gamma^{\infty}(W;G).$ Hence \eqref{eq:3.21} holds.

Assume \eqref{eq:3.21} holds at each point in $M$, it is obvious that  \eqref{eq:4.1} holds. This completes the proof.\qed
\end{proof}

In what follows we will investigate the problem of global controlled invariance for real analytic generalized affine distributions,
We have the following results where, instead of partition of unity used in the smooth case,  Cartan's Theorem A and B play a central role.
\begin{theorem}\label{tm:4.3}
Let $M$ be a $C^\omega$-manifold and $D$ be an involutive regular $C^\omega$-distribution of $M$ such that $M/\mathfrak{F}$ is a $C^\omega$-manifold, where $\mathfrak{F}$ is the foliation defined by $D$. Let $\Delta$ be a $C^\omega$-generalized affine distribution of $M$ with associated $C^{\omega}$-generalized distribution $L(\Delta)=G$.  Assume that  $(G+D)/D$ is a regular subbundle of $TM/D$ and $G \cap D$ is a $C^{\omega}$-generalized  distribution of $M$. Then $D$ is globally controlled invariant for $\Delta$ if and only if it is locally controlled invariant for $\Delta$ at each point $p \in M$.
\end{theorem}

We will first introduce some basic elements of sheaf theory and Cartan's Theorem A and B  for coherent real analytic sheaf that we will use in the proof of Theorem \ref{tm:4.3}.
\begin{definition}
A presheaf of sets over $M$ is an assignment to each open set $U \subseteq M$ a set $F(U)$ and, to each pair of open sets  $V, U \subseteq M$ with $V \subseteq U$, a map $r_{U, V}: F(U)\rightarrow  F(V)$ called the restriction map, with these
assignments having the following properties:\\
1.  $r_{U, U}$ is the identity map.\\
2.  If $W, V, U \subseteq M$ are open with $W \subseteq V \subseteq U$, then $r_{U, W}=r_{V,W} \circ r_{U,V}$.

A presheaf of rings over $M$ is a presheaf $\mathscr{R}=(\mathscr{R}(U))_{U \text{open}}$ whose local sections are rings and for which the restriction maps $r_{U,V}: \mathscr{R}(U)\rightarrow \mathscr{R}(V)$,
$U, V \subseteq M$ open, $V \subseteq U$ are homomorphisms of rings.

Let $\mathscr{R} = (\mathscr{R}(U))_{U \text{open}}$ be a presheaf of rings over $M$, a presheaf of $\mathscr{R}$-module is a presheaf $\mathscr{F} = (\mathscr{F}(U))_{U \text{open}}$  of sets such that $\mathscr{F}(U)$ is a module over $\mathscr{R}(U)$ and such that the restriction maps $r_{U,V}^{\mathscr{R}}$ and $r_{U,V}^{\mathscr{F}}$ satisfy
\begin{eqnarray}
&&r_{U,V}^{\mathscr{F}}(s+t)=r_{U,V}^{\mathscr{F}}(s)+r_{U,V}^{\mathscr{F}}(t),  s, t \in \mathscr{F}(U), \nonumber\\
&&r_{U,V}^{\mathscr{F}}(fs)=r_{U,V}^{\mathscr{R}}(f)r_{U,V}^{\mathscr{F}}(s), f \in \mathscr{R}(U), s \in \mathscr{F}(U).\nonumber
\end{eqnarray}
\end{definition}
\begin{definition}
A presheaf $\mathscr{F} = (\mathscr{F}(U))_{U \text{open}}$  is a sheaf if\\
1. The presheaf $\mathscr{F}$ is separated: if $U \subseteq M$ is open, if $(U_a)_{a\in A}$ is an open covering of $U$,  and if $s, t \in \mathscr{F}(U)$ satisfy $r_{U,U_a}(s) = r_{U,U_a}(t)$ for every $a  \in  A$.  then $s=t$.\\
2. The presheaf  $\mathscr{F}$ has the gluing property when, if $U \subseteq M$ is open, if $(U_a)_{a\in A}$ is an open covering of $U$, and if, for each $a \in A$, there exists $s_a \in  \mathscr{F}(U_a)$ with the family
$(s_a)_{a\in A}$ satisfying
\[r_{U_{a_1}, U_{a_1}\cap U_{a_2}} (s_{a_1})=r_{U_{a_2}, U_{a_1}\cap U_{a_2}}
(s_{a_2}),\]
for each $a_1, a_2 \in A$, then there exists $s \in  \mathscr{F}(U)$ such that $s_a = r_{U,U_a}(s)$ for each $a \in A$.
\end{definition}

Let $\mathscr{F}$ be a presheaf of sets on $M$. Let $p\in M$ and let $\mathcal{N}_p$  be the collection of open subsets of $M$ containing $p$. We define an equivalence relation in $(\mathscr{F}(U))_{U\in \mathscr{N}_p}$ by saying that $s_1 \in  \mathscr{F}(U_1)$ and $s_2 \in \mathscr{F}(U_2)$ are equivalent if there exists $\mathcal{V}\in  \mathscr{N}_p$ such that $V \subseteq U_1, V \subseteq U_2$, and $r_{U_1,V}(s_1) = r_{U_2, V}(s_2)$. The equivalence class of a section $s \in \mathscr{F}(U)$ we denote  by $[s]_p$.
\begin{definition}
 Let $\mathscr{F}=(\mathscr{F}(U))_{U \text{open}}$ be a presheaf of sets over $M$. For $p \in M$, the stalk of $\mathscr{F}$ at $p$ is the set of equivalence classes under the equivalence relation defined above, and is denoted by $\mathscr{F}_p$. The equivalence class $[s]_p$ of a section $s \in \mathscr{F}(U)$ is called the germ of $s$ at $p$.
\end{definition}

Then a natural subsheaf arising from a generalized subbundle is given as follows.
\begin{definition}
Let $\rho: E \rightarrow M$ be a $C^{r}$-vector bundle. Let $F \subseteq E$ be a $C^r$-generalized
subbundle.  The sheaf of sections of $F$ is the sheaf $\mathscr{G}_F^r$ whose local sections over the open set $U \subseteq M$ is the set of $C^r$-sections of $F|U$.
\end{definition}

$\mathscr{G}_F^r$ is a sheaf of modules over the sheaf of rings $\mathscr{C}_M^r$, where $\mathscr{C}_M^r$ is a sheaf of rings of $C^r$-functions on $M$. The stalk of $\mathscr{G}_F^r$ at $p \in M$ is a module  over the ring $\mathscr{C}_{p, M}^r$.

In the holomorphic case, the big result proved by Cartan \cite{cartan} and known as "Cartan's Theorem A", has as a consequence that the module of germs at $p$ of sections of a holomorphic vector bundle over a Stein base is generated by germs of
global sections. In \cite{Cartan} these holomorphic results are extended to the real analytic
case. However, Cartan's results extend far beyond
sheaves of sections of vector bundles to the setting of coherent analytic sheaves.
We refer to \cite{Lewis} for a detailed discussion on  Cartan's Theorem A and B for coherent real analytic sheaf, and hence for sheaf of sections of real analytic generalized distribution.
\begin{definition}
Let $M$ be a real analytic manifold. A
coherent real analytic sheaf is a coherent sheaf $\mathscr{F}$ of $\mathscr{C}_M^{\omega}$-modules.
\end{definition}
\begin{proposition}\cite{Lewis}\label{prop:2.1}
Let $M$ be a real analytic manifold. Let $G$ be a real analytic generalized distribution of $M$. Then $\mathscr{G}_G^{\omega}$ is a coherent real analytic sheaf.
\end{proposition}
\begin{theorem}\cite{Lewis}(Cartan's Theorem A)\label{tm:2.1}
Let $M$ be a real analytic
manifold and let $\mathscr{F}$ be a coherent real analytic sheaf. Then, for $p \in M$, the $\mathscr{C}_{p,M}^{\omega}$-module
$\mathscr{F}_p$ is generated by germs of global sections of $\mathscr{F}$.
\end{theorem}
Theorem \ref{tm:2.3} in Section \ref{sec:2} is a consequence of Proposition \ref{prop:2.1} and Theorem \ref{tm:2.1}.
\begin{theorem}\cite{Lewis}(A consequence of Cartan's Theorem B)\label{tm:2.2}
If $M$ is a real analytic manifold, if $\mathscr{F}$ is a coherent sheaf of $\mathscr{C}_{M}^{\omega}$-modules, and if
$\mathscr{U}=(U_a)_{a\in A}$ is an open cover of $M$, then $H^1(\mathscr{U};\mathscr{F})=0$, where $H^1(\mathscr{U};\mathscr{F})$ is the first
cohomology group of $\mathscr{F}$ for the cover $\mathscr{U}$.
\end{theorem}

Proof of Theorem \ref{tm:4.3}.
We only need to prove the "if" part. Since $D$ is locally controlled invariant for $\Delta$,  we know that $T\Phi(G)$ is a well-defined $C^{\omega}$-generalized  distribution of $M/\mathfrak{F}$ as claimed in Theorem \ref{tm:4.1}. Given $p \in M$ and $v \in G_p$,  it follows from Theorem \ref{tm:2.3} that there exists $t \in \Gamma^{\omega}(T\Phi(G))$ such that $t(\Phi(p))=T_{p}\Phi(v)$. We claim that there exists
$\hat s\in \Gamma^{\omega}(G)$ such that $T\Phi(\hat s)=t$.  By using similar arguments as given in Theorem \ref{tm:4.1}, we know that  for each $q \in M$, there exists $X_q \in \Gamma^{\omega}(U_q;G)$ such that $T\Phi(X_q)=t|\Phi(U_q)$, where $U_q$ is a local neighborhood  of $q$. Then we get an open covering $\{U_q\}$ of $M$.
If $U_x \cap U_y \neq 0, x, y \in M$, consider
\[X_{xy}=X_x|U_x \cap U_y-X_y|U_x \cap U_y.\]
Since $T\Phi(X_x|U_x \cap U_y-X_y|U_x \cap U_y)=0$, it follows that
\[X_x|U_x \cap U_y-X_y|U_x \cap U_y \in \Gamma^{\omega}(U_x \cap U_y; D)\cap \Gamma^{\omega}(U_x \cap U_y; G).\]
On the other hand, since
$\mathscr{G}^{\omega}_D$ and $\mathscr{G}^{\omega}_G$ are both coherent real analytic sheaves, it follows that $\mathscr{G}^{\omega}_D \cap \mathscr{G}^{\omega}_G$ is a coherent real analytic sheaf.
Then according to Theorem \ref{tm:2.2}, we have
 \[\check{\mathrm{H}}^1(\mathscr{U}; \mathscr{G}^{\omega}_D \cap \mathscr{G}^{\omega}_G)=0,\]
where $\check{\mathrm{H}}^1(\mathscr{U}; \mathscr{G}^{\omega}_D \cap \mathscr{G}^{\omega}_G)$ is the first cohomology group of the sheaf $\mathscr{G}^{\omega}_D \cap \mathscr{G}^{\omega}_G$ for the cover $\{U_q\}$.
Then it follows  from   $(X_{xy})_{x,y \in M} \in \check{\mathrm{Z}}^1(\mathscr{U};  \mathscr{G}^{\omega}_{D}\cap  \mathscr{G}^{\omega}_{G})$ that there exist $(\hat {X}_x)_{x\in M}$, where $\hat{X}_x \in (\mathscr{G}^{\omega}_D \cap \mathscr{G}^{\omega}_G)(U_x)$ such that $X_{xy}=\hat{X}_y-\hat{X}_x$ for $x, y \in M, U_x \cap U_y \neq 0$.\\
Let \[\hat s_x=X_x+\hat{X}_x, x \in M.\]
We have $\hat s_x|U_x \cap U_y=\hat s_y|U_x \cap U_y$. Hence we get a $C^\omega$-vector field $\hat s \in \Gamma^{\omega}(G)$ on $M$ such that $T\Phi(\hat s)=t$. This completes the proof of the claim.

Since $T_{p}\Phi(v-\hat s(p))=0$, we have $v-\hat s(p)\in D_p\cap G_p$. Since $G \cap D$ is  a $C^{\omega}$-generalized distribution of $M$, it follows from Theorem \ref{tm:2.3} that there exists a vector field $Y \in \Gamma^{\omega}(G \cap D)$, such that $Y(p)=v-\hat s(p)$. Now consider the vector field
$Y+\hat s\in \Gamma^{\omega}(G)$, we have $T\Phi(Y+\hat s)=t$ which yields that $[Y+\hat s, \Gamma^{\omega}(D)]\subseteq \Gamma^{\omega}(D)$.
Besides, $(Y+\hat s)(p)=v$. Hence \eqref{eq:1.7b} follows immediately since $p \in M$ and $v\in G_p$ are arbitrary.

To prove \eqref{eq:1.7a}, first we claim that there exists $\bar t \in \Gamma^{\omega}(T\Phi(\Delta))$.
Since for each $\bar p \in M/\mathfrak{F}$, there exist a local neighborhood $W_{\bar p}$ of $\bar p$
and $t_{\bar p} \in  \Gamma^{\omega}(W_{\bar p}; T\Phi(\Delta))$,  we get an open covering $\{W_{\bar p}\}$ of $ M/\mathfrak{F}$. If $W_{\bar x} \cap W_{\bar y} \neq 0, \bar x, \bar y \in M/\mathfrak{F}$, consider $t_{\bar x \bar y}=t_{\bar x}|W_{\bar x} \cap W_{\bar y}-t_{\bar y}|W_{\bar x} \cap W_{\bar y}$. It follows that $t_{\bar x\bar y}\in \Gamma^{\omega}(W_{\bar x} \cap W_{\bar y}; T\Phi(G))$. On the other hand, since
$\mathscr{G}^{\omega}_{T\Phi(G)}$ is a coherent real analytic sheaf, according to Theorem \ref{tm:2.2} we have  $\check{\mathrm{H}}^1(\mathscr{W}; \mathscr{G}^{\omega}_{T\Phi(G)})=0$. Then it follows  from   $(t_{\bar x\bar y})_{\bar x,\bar y \in M/\mathfrak{F}} \in \check{\mathrm{Z}}^1(\mathscr{W}; \mathscr{G}^{\omega}_{T\Phi(G)})$ that there exist $(\hat {t}_{\bar x})_{\bar x\in M/\mathfrak{F}}$, where $\hat{t}_{\bar x} \in \mathscr{G}^{\omega}_{T\Phi(G)}(W_{\bar x})$ such that $t_{\bar x \bar y}=\hat{t}_{\bar y}-\hat{t}_{\bar x}$ for $\bar x, \bar y \in M/\mathfrak{F}, W_{\bar x} \cap W_{\bar y} \neq 0$. Let $\bar t_{\bar x}=t_{\bar x}+\hat{t}_{\bar x}, \bar x \in M/\mathfrak{F}$.  We have $\bar t_{\bar x}|W_{\bar x} \cap W_{\bar y}=\bar t_{\bar y}|W_{\bar x} \cap W_{\bar y}$. Hence we get a $C^\omega$-vector field $\bar t \in \Gamma^{\omega}(T\Phi(\Delta))$.

Since $(G+D)/D$ is regular, for each $p \in M$, there exist  a neighborhood $U$ of $p$ and $f_p \in \Gamma^{\omega}(U; \Delta)$, such that $T\Phi(f_p)=\bar t|\Phi(U)$. By using similar arguments as given above, we know there exists $f \in \Gamma^{\omega}(\Delta)$, such that $T\Phi(f)=\bar t$, which yields that $[f, \Gamma^{\omega}(D)] \subseteq \Gamma^{\omega}(D)$. This completes the proof. \qed

A special case that doesn't require the condition of paracompact Hausdorff manifold is the trivial case presented as a main result in \cite{1}, where for each $p \in M$, $T_p\Phi: G(p) \rightarrow T_p\Phi(G(p))$ is one-to-one.
\begin{proposition}\cite{1}\label{pro:4.6}
Consider the driftless control system
\begin{equation}\label{eq:4.3}
\dot x=\sum_{i=1}^m g_i(x)u_i,
\end{equation}
on a $C^r$-manifold $M$ (not necessary paracompact and Hausdorff), where the vector fields $g_1, \cdots, g_m$ are linearly independent on $M$. Let $D$ be an involutive regular $C^r$-distribution of $M$ such that $M/\mathfrak{F}$ is a $C^r$-manifold (not necessary paracompact and Hausdorff). Assume that $G\cap D=0$. Assume also $D$ is locally controlled invariant for \eqref{eq:4.3} at each $p\in M$. Then $D$ is globally controlled invariant for \eqref{eq:4.3} if and only if $T\Phi(G)$ is trivial, where $\Phi: M \rightarrow M/\mathfrak{F}$ is the projection.
\end{proposition}

However, if there exists a drift term $f$ in \eqref{eq:4.3}, then the condition of Hausdorff and paracompact manifold is still needed. This is because if $M/\mathfrak{F}$
is not assumed to be Hausdorff and paracompact, then from conditions given in Proposition \ref{pro:4.6} we can not ensure the existence of $\bar t \in \Gamma^{r}(T\Phi(\Delta))$, which is necessary for \eqref{eq:1.6a}.

The $C^{\infty}$-part of the following theorem on global controlled invariance for affine control systems is taken from \cite{1}.
\begin{theorem}
Consider the affine control system
\begin{equation}\label{eq:4.4}
\dot x=f(x)+\sum_{i=1}^m g_i(x)u_i,
\end{equation}
on a $C^r$-manifold $M$, where $f, g_1, \cdots, g_m$ are $C^r$-vector fields on $M$ and $g_1, \cdots, g_m$ are linearly independent. Let $D$ be an involutive regular $C^r$-distribution of $M$ such that $M/\mathfrak{F}$ is a $C^r$-manifold. Assume that $G\cap D=0$. Assume also $D$ is locally controlled invariant for \eqref{eq:4.4}. Then $D$ is globally controlled invariant for \eqref{eq:4.4} if and only if $T\Phi(G)$ is trivial, where $\Phi: M \rightarrow M/\mathfrak{F}$ is the projection.
\end{theorem}
\begin{proof}
The proof for $r=\infty$ is given in \cite{1}. Let $r=\omega$. Since $M/\mathfrak{F}$ is a paracompact, Hausdorff $C^\omega$ -manifold, it follows from the Grauert-Morrey embedding theorem \cite{Grauert} that there exists real analytic Riemannian metric on $M/\mathfrak{F}$. Hence the result follows by replacing $r=\infty$ with $r=\omega$ in the proof of Proposition 5.1 and Proposition 5.2 in \cite{1}. \qed
\end{proof}
\begin{remark}
It should be pointed out that the proof for \eqref{eq:1.6a} in Proposition 5.1 in \cite{1} relies on the assumptions that  $G, T\Phi(G)$ are trivial and $G\cap D=0$.
\end{remark}
\subsection{Manifolds with a symmetry Lie group action}
Let $P$ be a smooth manifold. Let
\begin{equation}
\Phi: G\times P\rightarrow P: (g,p)\rightarrow \Phi_g(p)=gp
\end{equation}
be a proper action of a connected Lie group $G$ on $P$ (see \cite{Sniatycki}). The isotropy group $G_p$ of a point $p \in P$ is
$G_p=\{g\in G | gp=p\}$.
\begin{definition}\cite{Sniatycki}
A slice through $p \in P$ for an action of $G$ on $P$ is a submanifold $S_p$ of $P$ containing $p$ such that:\\
1. $S_p$ is transverse and complementary to the orbit $Gp$ of $G$ through $p$. In other words,
\[T_pP=T_pS_p\oplus T_p(Gp).\]\\
2. For every $p'\in S_p$, the manifold $S_p$ is transverse to the orbit $Gp'$; that is,
\[T_{p'}P=T_{p'}S_p+T_{p'}(Gp').\]
3. $S_p$ is $G_p$-invariant.\\
4. Let $p'\in S_p$. If $gp' \in S_p$, then $g \in G_p$.
\end{definition}

The existence of a slice through $p \in P$ is ensured by the following result.
\begin{proposition}\cite{Sniatycki}
There is an open ball $B$ in $\text{hor}T_pP$ centred at $0$ such that $S_p=Exp_p(B)$
is a slice through $p$ for the action of $G$ on $P$, where $Exp_p(v)$ is the value
at 1 of the geodesics of $G$-invariant Riemannian metric originating from $p$ in the direction $v$.
Further, the set $GS_p=\{gq | g\in G\, \text{and}\, q\in S_p\}$
is a G-invariant open neighbourhood of p in P.
\end{proposition}

A differential structure on a topological space $S$ is a family $C^{\infty}(S)$ of real-valued functions on $S$ satisfying the following conditions:\\
1. The family
\[\{f^{-1}(I)|f\in C^{\infty}(S) \,\text{and} \, I \, \text{is an open interval in} \, \mathbb{R}\}\]
is a subbasis for the topology of $S$.\\
2. If $f_1, \cdots, f_n \in C^{\infty}(S)$ and $F \in C^{\infty}(\mathbb{R}^n)$, then $F(f_1, \cdots, f_n) \in C^{\infty}(S)$.\\
3. If $f: S\rightarrow \mathbb{R}$ is a function such that, for every $x \in S$, there exist an open neighborhood $U$ of $x$,
and a function $f_x \in C^{\infty}(S)$ satisfying
\[f_x|U=f|U,\]
then $f \in C^{\infty}(S)$. Here, the subscript vertical bar $|$ denotes a restriction. A topological space S endowed with a differential
structure is called a differential space. A differential space $S$ is subcartesian if it is Hausdorff and every point $x \in S$ has a  neighbourhood $U$ diffeomorphic to a subset $V$ of $\mathbb{R}^n$.

For the orbit space $R=P/G$, let
\begin{equation}\label{eq:4.7}
C^{\infty}(R)=\{f:R\rightarrow \mathbb{R}|\rho^*f\in C^{\infty}(P)\},
\end{equation}
where $\rho: P\rightarrow R$ is the canonical projection (the orbit map). We have
\begin{proposition}\cite{Sniatycki}\label{prop:4.1}
The  topology of $R$ induced by the differential structure $C^{\infty}(R)$ coincides with the
quotient topology. Further, $R$ is a subcartesian space.
\end{proposition}

For locally compact, second countable Hausdorff differential spaces, there exists  a partition of unity.
\begin{definition}
A  countable partition of unity on a differential space $S$ is a countable family of functions $\{f_i\} \in C^{\infty}(S)$
such that:\\
(a) The collection of their supports is locally finite.\\
(b) $f_i(x)\geq 0$ for each $i$ and each $x \in S$.\\
(c) $\sum_{i=1}^{\infty}f_i(x)=1$ for each $x \in S$.
\end{definition}
\begin{theorem}\cite{Sniatycki}\label{tm:4.6}
Let $S$ be a differential space with differential structure $C^{\infty}(S)$,
and let $\{U_\alpha\}$ be an open cover of $S$. If $S$ is Hausdorff, locally compact and second
countable, then there exists a countable partition of unity $\{f_i\} \in C^{\infty}(S)$, such that, for each $i$, there exists $\alpha$ such that the support of $f_i$ is contained in $U_\alpha$ and the support of each $f_i$ is compact.
\end{theorem}

Since a subcartesian space is always Hausdorff and locally compact, to ensure existence of  partition of unity, we need to assume that the subcartesian space is second countable.

Now consider the tangent distribution $\mathcal{V}$ spanned by the fundamental vector fields of the action
of $G$ on $P$. At every point $p \in P$ it is defined by
\begin{equation}\label{eq:4.5}
\mathcal{V}_p=\{\xi_P(p)|\xi \in \mathfrak{g}\}.
\end{equation}
$\mathcal{V}$ is a smooth involutive generalized distribution of $P$. We have
\begin{theorem}\label{tm:4.4}
Let $G$ be a connected Lie group acting in a proper way on a second countable smooth
manifold $P$. Let $\Delta$  be  a smooth generalized affine distribution of $P$ with associated smooth generalized distribution $L(\Delta)=\mathcal{J}$ satisfying that for each $p \in P$, vectors in $\mathcal{J}_p$ are $G_p$-invariant. Assume that  \[[\Gamma^{\infty}(\Delta), \Gamma^{\infty}(\mathcal{V})] \subseteq \Gamma^{\infty}(\mathcal{J}),\]
where $\mathcal{V}$ is defined by \eqref{eq:4.5}. Then $\mathcal{V}$ is globally controlled invariant for the generalized affine distribution $\Delta$.
\end{theorem}
\begin{proof}
Since  $\mathcal{J}$ is smooth,  then for each $p \in P$ and $v_p \in \mathcal{J}_p$, there exists $X \in \Gamma^{\infty}(\mathcal{J})$ such that $X(p)=v_p$. By averaging $X$ over $G_p$, we get a $G_p$-invariant vector field
\[\tilde {X}=\int_{G_p} (\Phi_g)_{*}X d\mu(g),\]
where $d\mu(g)$ is the Haar measure on $G_p$ normalized so that $\text{vol}G_p=1$. Since $[\Gamma^{\infty}(\mathcal{J}), \Gamma^{\infty}(\mathcal{V})] \subseteq \Gamma^{\infty}(\mathcal{J})$, it follows that $(\Phi_g)_*v_q \in \mathcal{J}_{gq}$ for each $v_q \in \mathcal{J}_q$ and each $g \in G$. Hence we have $\tilde {X} \in \Gamma^{\infty}(\mathcal{J})$. Besides, since vectors in  $\mathcal{J}_p$ are $G_p$-invariant, we have
\[\tilde X(p)=\int_{G_p} (\Phi_g)_{*}v_p d\mu(g)=\int_{G_p} v_p d\mu(g)=v_p.\]
Now we can define a $G$-invariant vector field $\hat X$ on the $G$-invariant open neighborhood $GS_p$ as follows. For each $p'' \in GS_p$, there exists $g \in G$ such that $p''=gp'$ for $p' \in S_p$, and we set
\[\hat X(p'')=(\Phi_g)_{*}\tilde X(p').\]
Since $\tilde X$ is $G_p$-invariant, it follows that $\hat X$ is well-defined and $\hat X \in \Gamma^{\infty}(GS_p; \mathcal{J})$. On the other hand, from the proof of Proposition 4.3.2 in \cite{Sniatycki} we know that there exists a $G$-invariant smooth function $f$ on $P$ such that $f(p)=1$, and the support of $f$ is contained in $GS_p$ such that $P\backslash\text{supp}{f}$ is an open subset of $P$. Now consider the $G$-invariant vector field $f\hat X \in \Gamma^{\infty}(GS_p; \mathcal{J})$. It can be extended to a smooth $G$-invariant vector field $X'$ on $P$, which vanishes on $P\backslash\text{supp}{f}$. We have $X' \in \Gamma^{\infty}(\mathcal{J})$, and $X'(p)=f(p)\hat X(p)=\tilde X(p)=v_p$. Since $X'$ is $G$-invariant, it follows that $[X', \Gamma^{\infty}(\mathcal{V})] \subseteq \Gamma^{\infty}(\mathcal{V})$. Hence $\mathcal{V}$ is globally controlled invariant for  $\mathcal{J}$ since $p \in P$ and $v_p \in \mathcal{J}_p$ are arbitrary.

Let $Y \in \Gamma^{\infty}(\Delta)$.  Let $p \in P$. By averaging $Y$ over $G_p$, we get a $G_p$-invariant vector field
\[\tilde Y=\int_{G_p} (\Phi_g)_{*}Y d\mu(g).\]
Since $[Y, \Gamma^{\infty}(\mathcal{V})] \subseteq \Gamma^{\infty}(\mathcal{J})$, it follows that $(\Phi_g)_*v_q \in {\Delta}_{gq}$ for each $v_q \in \Delta_q$ and each $g \in G$, which yields that $\tilde{Y} \in \Gamma^{\infty}(\Delta)$, since $\text{vol}G_p=1$. By using similar arguments as above,  $\tilde Y$ can be extended to a $G$-invariant vector field $\hat Y_p\in \Gamma^{\infty}(GS_p; \Delta)$. Consider the open cover $\{GS_p|p \in P\}$ of $P$. It follows that $\{\pi(GS_p)|p \in P\}$ is an open cover of $P/G$, where $\pi: P \rightarrow P/G$ is the orbit map. Since $P$ is second countable, it follows that $P/G$ with the quotient topology is second countable. Then from Proposition \ref{prop:4.1} and Theorem \ref{tm:4.6}
we know that there exists a countable partition of unity $\{\bar {f}_i\} \in C^{\infty}(P/G)$,
subordinate to $\{\pi(GS_p)|p \in P\}$  and such that the support of each $\bar f_i$ is compact. Now we define the vector field $Y'$ on $P$ by
\[Y'=\sum_i\pi^*\bar f_i\hat Y_{\alpha(i)},\]
where $\alpha(i) \in P$ is such that the support of $\bar f_i$ is contained in $\pi(GS_{\alpha(i)})$. $Y'$ is a smooth $G$-invariant vector field on $P$. Hence $[Y', \Gamma^{\infty}(\mathcal{V})] \subseteq \Gamma^{\infty}(\mathcal{V})$. It remains to show that $Y'(p)\in \Delta(p)$ for each $p \in P$. Actually we have a finite sum at each $p \in P$:
\begin{align*}
Y'(p)&=\sum_{i}(\pi^*\bar f_i)(p)\hat Y_{\alpha(i)}(p)=\sum_{i}\bar f_i(\pi(p))(v_0+v_{p_i})\\
&=(\sum_{i}\bar f_i(\pi(p)))v_0+\sum_{i}\bar f_i(\pi(p))v_{p_i}=v_0+v \in \Delta_p,
\end{align*}
where $v_0 \in \Delta_p, v_{p_i}, v \in \mathcal{J}_p$. This completes the proof. \qed
\end{proof}
If the action of $G$ on $P$ is free and proper, then the tangent distribution $\mathcal{V}$ is regular. We have the following result for global controlled invariance of the involutive regular distribution $\mathcal{V}$.
\begin{theorem}\label{tm:4.5}
Let $G$ be a connected Lie group acting in a free and proper way on a
second countable smooth manifold $P$. Let $\Delta$  be  a smooth generalized affine distribution of $P$ with associated smooth generalized distribution $L(\Delta)=\mathcal{J}$. Assume that
\begin{equation}\label{eq:4.6}
[\Gamma^{\infty}(\Delta), \Gamma^{\infty}(\mathcal{V})] \subseteq \Gamma^{\infty}(\mathcal{J}),
\end{equation}
where $\mathcal{V}$ is defined by \eqref{eq:4.5}.
Then $\mathcal{V}$ is globally controlled invariant for the generalized affine distribution $\Delta$.
\end{theorem}
\begin{proof}
Since the action of $G$ on $P$ is free, it follows that $G_p=e$ for each $p \in P$. Hence vectors in  $\mathcal{J}_p$ are  $G_p$-invariant, for each $p \in P$. Then the result follows immediately from Theorem \ref{tm:4.4}.\qed
\end{proof}
\begin{remark}
In Theorem \ref{tm:4.1}, it is required that $(G+D)/D$ is regular. From \eqref{eq:4.6} we know that $\mathcal{V}\subseteq \mathcal{J}$.
So in the setting of Theorem \ref{tm:4.5},  the above condition is equivalent with requiring that $\mathcal{J}$ is regular,
which is not satisfied in Theorem \ref{tm:4.5}. Despite this,  the result for global controlled invariance still holds as claimed in Theorem \ref{tm:4.5}. This is due to the special geometric structure of manifolds with a symmetry Lie group action.
\end{remark}

\section{Conclusions}
\label{sec:5}
In this paper, we have investigated the problem of controlled invariance of involutive regular distributions for both  affine control systems and generalized affine distributions. A complete characterization of local controlled invariance for both smooth and real analytic affine control systems has been given. We have provided a differential geometric proof for this local result. A sufficient condition for local controlled invariance for real analytic generalized affine distribution has  been presented. We have given sufficient conditions that allow passages from local to  global controlled invariance for both smooth and real analytic generalized affine distribution of general manifolds. Existing results in the literature have been reviewed and clarified. Finally, for smooth manifolds with a proper Lie group action,
the problem of global controlled invariance has  been investigated, where the controlled invariant distribution is not required to be regular. As mentioned in the introduction, the problem of controlled invariance of singular distribution is more complicated and yet has not been solved well.
It will be addressed in our future work.

\begin{acknowledgements}
This work was performed while the author was visiting the Department of Mathematics and Statistics, Queen's University. The  author would like to thank Professor Andrew Lewis for some enlightening talks, for example on the use of Cartan's Theorem B. The author would also like to thank the anonymous reviewers for their careful reading of the manuscript and  a lot of helpful and constructive suggestions which help improve the presentation of the paper a lot. This work was supported by CSC and the NSFC Grant 61703211.
\end{acknowledgements}



\end{document}